\documentclass{birkjour}
\usepackage[noadjust]{cite}
\usepackage{subcaption} 
\usepackage{hyperref}
\usepackage[dvipsnames]{xcolor}
\newtheorem{thm}{Theorem}[section]
\newtheorem{cor}[thm]{Corollary}
\newtheorem{lem}[thm]{Lemma}
\newtheorem{prop}[thm]{Proposition}
\theoremstyle{definition}
\newtheorem{defn}[thm]{Definition}
\theoremstyle{remark}
\newtheorem{rem}[thm]{Remark}

\newtheorem*{ex}{Example}
\numberwithin{equation}{section}
\newcommand{\BibTeX}{\textsc{B\kern-0.1emi\kern-0.017emb}\kern-0.15em\TeX}
\newcommand{\ed}{\end{document}}
\begin{document}
\title[Constructing $3$-Dimensional Monogenic Homogeneous Functions]{Constructing $3$-Dimensional Monogenic Homogeneous Functions}
\author[H. Baghal Ghaffari]{Hamed Baghal Ghaffari}
\address{School of Mathematical and Physical Sciences, \\University of Newcastle, \\Callaghan, NSW 2308, \\Australia}
\email{hamed.baghalghaffari@newcastle.edu.au}
\author[J.A. Hogan]{Jeffrey A. Hogan}
\address{School of Mathematical and Physical Sciences, \\University of Newcastle, \\Callaghan, NSW 2308, \\Australia}
\email{jeff.hogan@newcastle.edu.au}
\author[J.D. Lakey]{Joseph D. Lakey}
\address{Department of Mathematical Sciences,\\New Mexico State University, \\
Las Cruces, NM 88003–8001, USA}
\email{jlakey@nmsu.edu}

\date{\today}
\begin{abstract}
This paper is dedicated to the construction of multidimensional spherical monogenics. Firstly, we investigate the construction of monogenic functions in dimension $3$ by applying the Dirac operator to the orthonormal bases of spherical harmonics, resulting in orthogonal spherical monogenics. Additionally, we employ the reproducing kernel for monogenic functions and a specialized optimization method to derive various types of $3$-dimensional spherical harmonics and spherical monogenics.
\end{abstract}

\bigskip
\maketitle
\section{Introduction}
Clifford prolate spheroidal wave functions (CPSWFs), initially introduced in \cite{ghaffari2019clifford, ghaffari2022higher}, have shown to possess remarkable properties. These functions, generated as eigenfunctions of the extended Clifford-Legendre differential equations, exhibit a fascinating connection as eigenfunctions of the finite Fourier transformation across all dimensions. To explicitly determine the CPSWFs, it becomes essential to establish a clear basis of spherical monogenics. Therefore, this paper aims to construct a robust basis of spherical monogenics in dimension $3$.

This paper is structured as follows: In Section 2, we provide a comprehensive review of pertinent concepts in Clifford analysis. Moving to Section 3, we construct an orthonormal basis of spherical monogenics in $\mathbb{R}^{3}$ by employing the Dirac operator on an orthonormal basis of spherical harmonics. This approach draws inspiration from the works of \cite{caccao2004complete}, \cite{bock2010generalized}, and \cite{lavivcka2012complete}. In section 4, we use knowledge of the reproducing kernel for the spherical monogenics of degree $k$, optimization on the sphere and Clifford analysis techniques to construct an orthonormal basis of near-zonal spherical monogenics of degree $k$.

\section{Preliminaries}
Let
$\mathbb{R}^{m}$
be 
$m$-dimensional 
euclidean space and let
$\{e_{1},e_{2},\dots e_{m}\}$
be an orthonormal basis for
$\mathbb{R}^{m}.$
We endow these vectors with the multiplicative properties
\begin{align*}
	e_{j}^{2}&=-1,\; \; j=1,\dots , m,\\
	e_{j}e_{i}&=-e_{i}e_{j}, \;\; i\neq j, \;\; i,j=1,\dots , m.
\end{align*}
For any subset
$A=\{j_{1},j_{2},\dots, j_{h}\}\subseteq \{1,\dots ,	m\}=M,$ with $j_1<j_2<\cdots <j_h$
we consider the formal product
$e_{A}=e_{j_{1}}e_{j_{2}}\dots e_{j_{h}}.$
Moreover for the empty set
$\emptyset$
one puts
$e_{\emptyset}=1$ (the identity e{lem}ent). The Clifford algebra ${\mathbb R}_m$ is then the $2^m$-dimensional algebra 
$${\mathbb R}_m=\bigg\{\sum\limits_{A\subset M}\lambda_Ae_A:\, \lambda_A\in{\mathbb R}\bigg\}.$$
Every e{lem}ent $\lambda=\sum\limits_{A\subset M}\lambda_Ae_A\in{\mathbb R}_m$ may be decomposed as  
$\lambda=\sum\limits_{k=0}^{m}[\lambda]_{k},$
where 
$[\lambda]_{k}=\sum\limits_{\vert A\vert=k}\lambda_{A}e_{A}$
is the so-called 
$k$-vector
part of 
$\lambda\, (k=0,1,\dots ,m)$. Denoting by 
$\mathbb{R}_{m}^{k}$
the subspace of all 
$k$-vectors
in
$\mathbb{R}_{m},$
i.e., the image of 
$\mathbb{R}_{m}$
under the projection operator 
$[\cdot]_{k},$
one has the multi-vector decomposition
$\mathbb{R}_{m}=\mathbb{R}_{m}^{0}\oplus \mathbb{R}_{m}^{1}\oplus\cdots \oplus \mathbb{R}_{m}^{m},$
leading  to the identification of
$\mathbb{R}$
with the subspace of real scalars
$\mathbb{R}_{m}^{0}$
and of
$\mathbb{R}^{m}$
with the subspace of real Clifford vectors 
$\mathbb{R}_{m}^{1}.$ The latter identification is achieved by identifying the point
$(x_{1},\dots,x_{m})\in{\mathbb R}^m$
with the Clifford number
$x=\sum\limits_{j=1}^{m}e_{j}x_{j}\in{\mathbb R}_m^1$. Let $\mathbb{R}_{m}^{+}$ as the even subalgebra of $\mathbb{R}_{m}$, i.e., $\mathbb{R}_{m}^{+}=\bigg\{\sum\limits_{\vert A\vert\; even}\lambda_Ae_A:\, \lambda_A\in{\mathbb R} \bigg\}$.
The Clifford number 
$e_{M}=e_{1}e_{2}\cdots e_{m}$
is called the pseudoscalar; depending on the dimension 
$m,$
the pseudoscalar commutes or anti-commutes with the 
$k$-vectors
and squares to 
$\pm 1.$
The Hermitian conjugation is the real linear mapping $\lambda\mapsto\bar{\lambda}$ of ${\mathbb R}_m$ to itself satisfying
\begin{align*}
	\overline{\lambda \mu}&=\bar{\mu}\bar{\lambda},\;\;\;\; \textnormal{for all}\;\lambda,\mu\in\mathbb{R}_{m}\\
	\overline{\lambda\, e_{A}}&=\lambda\, \overline{e_{A}},\;\;\; \lambda\in\mathbb{R},\\
	\overline{e_{j}}&=-e_{j},\;\; j, \;\; j=1,\ldots , m.
\end{align*}
The Hermitian conjugation leads to a Hermitian inner product and its associated norm on 
$\mathbb{R}_{m}$
given respectively by
$$(\lambda, \mu)=[\bar{\lambda}\mu]_{0}\;\;\;\textnormal{and}\;\;\; \vert\lambda\vert^{2}=[\bar{\lambda}\lambda]_{0}=\sum\limits_{A}\vert\lambda_{A}\vert^{2}.$$
The product of two vectors splits up into a scalar part and a 2-vector, also called a bivector:
$$xy=-\langle x, y\rangle +x\wedge y$$
where
$\langle x,y\rangle=-\sum\limits_{j=1}^{m}x_{j}y_{j}\in \mathbb{R}^{0}_{m}$,
and
$x\wedge y=\sum\limits_{i=1}^{m}\sum\limits_{j=i+1}^{m}e_{i}e_{j}(x_{i}y_{j}-x_{j}y_{i})\in\mathbb{R}^{2}_{m}$.
Note that the square of a vector variable 
$x$
is scalar-valued. In fact, if $x\in \mathbb{R}_{m}^{1}$ then
$$x^{2}=-\langle x,x\rangle=-\vert x\vert^{2}.$$
By diect calculations we see that if $z\in \mathbb{R}_{3}$ then $\bar{z} z =\vert z\vert^{2}$.

Clifford analysis provides a function theory that serves as a higher-dimensional analogue of the theory of holomorphic functions in complex analysis. In this framework, functions are defined on the Euclidean space $\mathbb{R}^{m}$ and take values in the Clifford algebra $\mathbb{R}_{m}$. 

The central notion in Clifford analysis is monogenicity, which is a multidimensional counterpart of holomorphy in the complex plane. 
Let $\Omega\subset{\mathbb R}^m$, $f:\Omega\to{\mathbb R}_m$ and $n$ a non-negative integer. We say $f\in C^n(\Omega,{\mathbb R}_m )$ if $f$ and all its partial derivatives of order less than or equal to $n$ are continuous.
\begin{defn} Let $\Omega\subset{\mathbb R}^m$. 
	A function 
	$f\in C^1(\Omega ,{\mathbb R}_m)$
	is said to be left monogenic in that region if 
	$$\partial_{x}f=0.$$
	Here 
	$\partial_{x}$
	is the Dirac operator in
	$\mathbb{R}^{m}$, i.e.,
	$\partial_{x}f=\sum\limits_{j=1}^{m}e_{j}\partial_{x_{j}}f$,
	where
	$\partial_{x_{j}}$
	is the partial derivative
	$\dfrac{\partial}{\partial x_{j}}$.
	The notion of right monogenicity is defined in a similar way by letting the Dirac operator act from the right. It is easily seen that if a Clifford algebra-valued function 
	$f$
	is left monogenic, its Hermitian conjugate 
	$\bar{f}$
	is right monogenic. The Euler operator is defined on  
	$C^1(\Omega ,\mathbb{R}_{m})$
	by
	$E=\sum\limits_{j=1}^{m}x_{j}\partial_{x_{j}}$.
	If $k$ is a non-negative integer and $f\in C^1({\mathbb R}^m\setminus\{0\},{\mathbb R}_m)$ is homogeneous of degree $k$ (i.e., $f(\lambda x)=\lambda^kf(x)$ for all $\lambda >0$ and $x\in{\mathbb R}^m$) then $Ef=kf$.
	The Laplace operator is factorized by the Dirac operator as follows:
	\begin{equation}
		\Delta_{m}=-\partial_{x}^{2}.
	\end{equation}
\end{defn}

\begin{defn}\label{left monogenic homogeneous polynomial}
	A left monogenic homogeneous polynomial
	$Y_{k}$
	of degree 
	$k\; (k\geq 0)$
	on
	$\mathbb{R}^{m}$
	is called a left solid inner spherical monogenic of order 
	$k.$
	The right Clifford module of all left solid inner spherical monogenics of order 
	$k$
	will be denoted by
	$\mathcal{M}_{\ell}^{+}(k).$
	It can be shown 
	\cite{delanghe2012clifford}
	that the dimension of 
	$\mathcal{M}_{\ell}^{+}(k)$
	is given by 
	\begin{equation}\label{dimension_monogenic}
		\dim\mathcal{M}_{\ell}^{+}(k)=\frac{(m+k-2)!}{(m-2)!k!}=\binom{m+k-2}{k}=d_{k}^{m}.
	\end{equation}
	We may choose an orthonormal basis for each 
	$\mathcal{M}_{\ell}^{+}(k)$, $(k\geq 0)$ i.e., a collection $\{Y_{k}^{j}\}_{j=1}^{d_{k}^{m}}$ which spans $\mathcal{M}_{\ell}^+(k)$ and for which 
	$$\int\limits_{S^{m-1}}\overline{Y_{k}^{j}(\theta)}Y_{k}^{j'}(\theta)d\theta=\delta_{jj'}.$$
\end{defn}

We will need also the following Lemma which is easily to obtain by direct calculation.

\begin{thm}(\textbf{Clifford-Stokes Theorem} \cite{delanghe2012clifford})  \label{Clifford-Stokes theorem} 
	Let $\Omega\subset{\mathbb R}^m$
	$f,g\in {C}^{1}(\Omega),$
	and assume that 
	$C$
	is a compact orientable 
	$m-$dimensional manifold with boundary
	$\partial C$.
	Then for each
	$C\subset \Omega,$
	one has
	$$\int\limits_{\partial C}f(x)n(x)g(x)d\sigma(x)=\int\limits_{C}[(f(x)\partial_{x})g(x)+f(x)(\partial_{x}g(x))]dx,$$
	where 
	$n(x)$
	is the outward-pointing unit normal vector on
	$\partial C.$
\end{thm}
The following theorems give us the reproducing kernel of spherical harmonic and monogenic functions. Let $\mathcal{H}^{+}(k)$ be the space of homogeneous harmonic polynomials of degree $k$ on $\mathbb{R}^{m}$. Then $h_{k}^{+}=\dim (\mathcal{H}^{+}(k))=\frac{2k+m-2}{k+m-2}\binom{k+m-2}{m-2}$ is the dimension of $\mathcal{H}^{+}(k)$. Let $\{ H_{k}^{i} \}_{i=1}^{h_{k}^{m}}$ be any orthonormal basis for $\mathcal{H}^{+}(k)$.

\begin{thm}
		The function
	\begin{equation}\label{Repruducing_Kernel_of_harmonic_Functions}
		R_{k}^{m}(x,y)=\sum_{i=1}^{h_{m}^{k}} H_{k}^{i}(x)\overline{H_{k}^{i}(y)}=\frac{(2k+m-2)}{(m-2)} C_{k}^{\mu}( \langle x,y \rangle ).
	\end{equation}
	is the reproducing kernel for $\mathcal{H}^{+}(k)$ in the sense that $\int\limits_{S^{m-1}}R_{k}^{m}(x,y) H_{k}(y) d\sigma(y)=H_{k}(x)$ for all $H_{k}\in \mathcal{H}^{+}(k)$ and $x\in S^{m-1}$.
\end{thm}

\begin{thm}\label{reproducing_kernel_equation}
	Let $\{ Y_{k}^{i} \}_{i=1}^{d_{k}^{m}}$ be any orthonormal basis for $\mathcal{M}_{\ell}^{+}(k)$ and
	\begin{equation}\label{Repruducing_Kernel_of_Monogenic_Functions}
		K_{k}^{m}((x,y)=\sum_{i=1}^{d_{k}^{m}} Y_{k}^{i}(x)\overline{Y_{k}^{i}(y)}=\frac{(k+m-2)}{(m-2)} C_{k}^{\mu}(\frac{\langle x,y \rangle}{\vert x\vert \vert y\vert})+(x\wedge y) C_{k-1}^{\mu+1}(\frac{\langle x,y \rangle}{\vert x\vert \vert y\vert}),
	\end{equation}
where $C_{k}^{\mu}(t)$ is the Gegenbauer polynomials defined on the line with $\mu=\frac{m}{2}-1$. Then $K_{k}^{m}$ is the reproducing kernel for $\mathcal{M}_{\ell}^{+}(k)$ in the sense that $\int\limits_{S^{m-1}}K_{k}^{m}(x,y) Y_{k}(y) d\sigma(y)=Y_{k}(x)$ for all $Y_{k}\in \mathcal{M}_{\ell}^{+}(k)$ and $x\in S^{m-1}$.
\end{thm}
\begin{proof}
	For the proof see \cite{de2016reproducing} Theorem 3.3.
\end{proof}
\begin{lem}
The Dirac operator decomposes as
\begin{equation}\label{dirac_operator_polar_form_gamma}
	\partial_{x}f=\frac{x}{\vert x\vert}\big(\partial_{r}+\frac{1}{r}\Gamma\big)f,
\end{equation}
where $\Gamma$ is the spherical Dirac operator given by
\begin{equation}\label{Definition_Gamma}
	\Gamma = -\sum_{i<j}e_{ij}\bigg(x_{i}\frac{\partial}{\partial x_{j}}-x_{j}\frac{\partial}{\partial x_{i}}\bigg).
\end{equation}
\end{lem}

\begin{defn}
Let $f:S^{m-1} \to \mathbb{C} $. 
We define $F:\mathbb{R}^{m}\to \mathbb{C}$ by $F(x)=f(\frac{x}{\vert x\vert})\; (x\neq 0)$. 
The Laplace-Beltrami operator $\triangle_{T}$ can be defined as follows
\begin{equation}
\triangle_{T}f(x)=\triangle F(x)\big\vert_{x\in S^{m-1}}.
\end{equation}
\end{defn}
It is possible to see that \cite{delanghe2012clifford,craddock2013fractional} 
$$ \Gamma^{2}  = -\triangle_{T}+(m-2)\Gamma. $$
Furthermore, spherical harmonics are eigenfunctions of $\triangle_{T}$. In fact,
\begin{equation}\label{laplace_Beltrami_eigenvalue}
\triangle_{T}\psi = -k(k+m-2)\psi
\end{equation}
for all $\psi\in \mathcal{H}^{+}(k)$. Now here we want to review the definition of complex adjoint matrix from \cite{zhang1997quaternions}.
\begin{defn}\label{Definition_Complex_Adjoint}
Let $\mathbb{H}=\mathbb{R}_{2}$ be the algebra of (real) quaternions and
$$\mathbb{H}_{e}=sp_{\mathbb{R}}\{ 1,e_{12} \}$$ be the even subalgebra of $\mathbb{H}$, $\mathbb{H}_{e}$ is isomorphic to the complex numbers $\mathbb{C}$ and each $q\in \mathbb{H}$ decomposes as
$$q=q_1+q_2 e_2$$
with $q_1, q_2 \in \mathbb{H}_e$.	Similarly, each matrix $A\in M_{n\times n}(\mathbb{H})$ decomposes as 
\begin{equation} \label{Decomposition_Matrix_Subalgebra}
	A=A_{1}+A_{2}e_{2}
\end{equation}
with $A_{1},\, A_{2} \in M_{n\times n}(\mathbb{H}_{e})\cong M_{n\times n}(\mathbb{C})$. The complex adjoint $\chi(A)\in M_{2n\times 2n}(\mathbb{C})$ of $A$ is given by 
$$
\chi(A):=\begin{bmatrix}
A_{1} & A_{2} \\
-\overline{A_2} & \overline{A_1} \\
\end{bmatrix}
$$
with $A_{1},\, A_{2}  $ as in \eqref{Decomposition_Matrix_Subalgebra}.
\end{defn}
\begin{thm}\cite{lee1948eigenvalues}. \label{Lee_s_theorem_for_quaternionic_matrix}
Let $ A,B\in M_{n\times n}(\mathbb{H}) $ and $\chi:M_{n\times n}(\mathbb{H})\to M_{2n\times 2n}(\mathbb{C})$ the complex adjoint mapping. Then
\begin{enumerate}
\item[(i)]  $ \chi(I_{n})=I_{2n},$
\item[(ii)]  $\chi (AB)= \chi(A)\chi(B),$
\item[(iii)]  $\chi(A^{\ast})= (\chi(A))^{\ast},\;\; \textnormal{here $\ast$ means the conjugation transpose of a matrix,}$
\item[(iv)]  $\chi (A^{-1})=(\chi(A))^{-1}\;\; \textnormal{if}\; A^{-1}\; \textnormal{exist}.$
\end{enumerate}
\end{thm}
\begin{defn}
A self-adjoint matrix $ G\in M_{n\times n}(\mathbb{H}) $ is said to be positive semidefinite if $ x^{\ast}Gx\geq 0 $ for all $x\in\mathbb{H}^{n}$.
\end{defn}	
\begin{lem} \cite{zhang1997quaternions}. \label{selfadjoint by complex adjoint transformation}
$A\in M_{n\times n}(\mathbb{H})$ is a self-adjoint positive semidefinite matrix if and only if $\chi(A)\in M_{2n\times 2n}(\mathbb{C})$ is a self-adjoint positive semidefinite matrix.
\end{lem}

The Jacobi polynomials $P_n^{(\alpha ,\beta )}(x)$ ($n$ a nonnegative integer, $\alpha ,\beta >-1$, $x\in[-1,1]$) satisfy the differential equation
$$(1-x^2)y''+[\beta -\alpha -(\alpha +\beta+2)x]y'=-n(n+\alpha +\beta +1)y.$$
The eigenvalues $-n(n+\alpha +\beta +1)$ are non-degenerate. They are normalized in such a way as to admit the explicit representation
$$P_n^{(\alpha ,\beta )}(x)=2^{-n}\sum_{s=0}^n\binom{n+\alpha}{n-s}\binom{n+\beta}{s}(x-1)^s(x+1)^{n-s}.$$
From \cite{gradshteyn2007ryzhik} we have the following recurrence
\begin{align}
(2n+\alpha+\beta)(1-x^{2})\frac{d}{dx}P_{n}^{(\alpha,\beta)}(x)=&n\big[ (\alpha-\beta)-(2n+\alpha+\beta)x\big]P_{n}^{(\alpha,\beta)}(x)\nonumber\\
+&2(n+\alpha)(n+\beta)P_{n-1}^{(\alpha,\beta)}(x). 
\end{align}
By letting $x=\cos\theta$ we have that
\begin{align}
	-(2n+\alpha+\beta)\sin\theta\frac{d}{d\theta}(P_{n}^{(\alpha,\beta)}(\cos\theta))=&n\big[ (\alpha-\beta)-(2n+\alpha+\beta)\cos\theta\big]P_{n}^{(\alpha,\beta)}(\cos\theta)\nonumber\\
	+&2(n+\alpha)(n+\beta)P_{n-1}^{(\alpha,\beta)}(\cos\theta). \label{jecobi_derivative_recurrence}
\end{align}

\section{Constructing $3$-Dimensional Spherical Monogenics from \\ $3$-Dimensional Spherical Harmonics}
In this section, our objective is to construct $3$-dimensional spherical monogenic functions of degree $(k-1)$ by applying the Dirac operator to $3$-dimensional spherical harmonics of degree $k$. By \eqref{dimension_monogenic}, the dimension of the space $\mathcal{M}_{\ell}^{+}(k)$ of spherical monogenics of degree $k$ in $\mathbb{R}^{3}$ is $k+1$. Hence, our goal is to find $k+1$ linearly independent spherical monogenics of degree $k$. Let $x=(x_{1},x_{2},x_{3})=(r\cos\theta \sin\phi , r\cos\theta \sin\theta , r\cos\phi)\in \mathbb{R}^{3}$ with $r\geq 0,$ $0\leq \theta< 2\pi$ and $0\leq \phi\leq \pi$. To proceed with the construction, we use the particular orthonormal basis $\{H_{k}^{n}\}_{n=0}^{2k} $ of spherical harmonics of degree $k$ in $\mathbb{R}^{3}$ given by
\begin{align}
&H_{k}^{0}(x)=\frac{r^{k}}{\sqrt{8\pi}}P_{k}^{(0,0)}(\cos\phi)\label{zero_harmonic},\\
&H_{k}^{2n}(x)=\frac{r^{k}}{2^{n+1}\sqrt{\pi}}(\sin\phi)^{n}P_{k-n}^{(n,n)}(\cos\phi)\cos n\theta \;\;\;\; (1\leq n\leq k), \label{even_harmonic}\\
&H_{k}^{2n+1}(x)=\frac{r^{k}}{2^{n+1}\sqrt{\pi}}(\sin\phi)^{n}P_{k-n}^{(n,n)}(\cos\phi)\sin n\theta \;\;\;\; (1\leq n\leq k)\label{odd_harmonic}.
\end{align}
We recall that the $\dim(\mathcal{H}_{k}^{3})=2k+1$. By changing variables to spherical coordinates $(r,\theta,\phi),$ first order partial derivatives $  \partial_{x_{1}} f, \; \partial_{x_{2}}f,\;  \partial_{x_{3}}f $ may be written as follows 
$$\begin{bmatrix}
\partial_{x_{1}}  \\
\\
\partial_{x_{2}}\\
\\
\partial_{x_{3}}
\end{bmatrix}=\begin{bmatrix}
\cos\theta \sin\phi & -\dfrac{\sin\theta}{r\sin\phi} &\dfrac{\cos\theta\cos\phi}{r} \\
\sin\theta \sin\phi & \dfrac{\cos\theta}{r\sin\phi} &\dfrac{\sin\theta\cos\phi}{r} \\
\cos\phi & 0 & -\dfrac{\sin\phi}{r}
\end{bmatrix}
\begin{bmatrix}
\partial_{r}  \\
\\
\partial_{\theta} \\
\\
\partial_{\phi}
\end{bmatrix}.$$
In spherical coordinates, the spherical Dirac operator becomes
\begin{align}\label{Gamma_Spherical_Version}
\Gamma f=& -e_{12}\dfrac{\partial f}{\partial \theta}-e_{13}\bigg(-\cos\theta\, \dfrac{\partial f}{\partial \phi}+\cot\phi\, \sin\theta\, \dfrac{\partial f}{\partial \theta}\bigg)-e_{23}\bigg(-\sin\theta\,\dfrac{\partial f}{\partial \phi}-\cot\phi \,\cos\theta\,\dfrac{\partial f}{\partial \theta}\bigg)\nonumber\\
=&\frac{1}{\sin\phi}\, e_{23}e^{-e_{12}\theta}e^{\phi e_{13}e^{-e_{12}\theta}}\, \frac{\partial f}{\partial \theta}+e_{13}e^{-e_{12}\theta}\,\frac{\partial f}{\partial\phi},
\end{align}
where $e^{e_{ij}\theta}=\cos\theta+e_{ij}\sin\theta$.
\begin{lem}\label{conjugation_ephi}
	Let $\theta$, $\phi$ be real numbers. Then 
	$$\overline{e^{-\phi\, e_{13} e^{-e_{12}\theta}}}=e^{\phi\, e_{13} e^{-e_{12}\theta}}.$$
\end{lem}
\begin{proof}
	Since $(e_{13}\cos\theta+e_{23}\sin\theta)^{2}=-1$, direct calculation gives
	\begin{align*}
		e^{-\phi\, e_{13} e^{-e_{12}\theta}}&=e^{-\phi\, e_{13} [\cos\theta -e_{12}\sin\theta]}\\
		&=e^{-\phi\,  [e_{13}\, \cos\theta +e_{23}\, \sin\theta]}=\cos\phi-[e_{13}\, \cos\theta +e_{23}\,\sin\theta] \sin\phi,
	\end{align*}
	Hence,
\begin{align*}
	\overline{e^{-\phi\, e_{13} e^{-e_{12}\theta}}}&=\cos\phi+[e_{13}\, \cos\theta +e_{23}\,\sin\theta] \sin\phi\\
	&=\cos\phi+e_{13}[\cos\theta -e_{12}\sin\theta] \sin\phi=e^{\phi\, e_{13} e^{-e_{12}\theta}}.
	\end{align*}
\end{proof}
\begin{lem}\label{Commutation_Lemma}
	Let $\theta,\,\phi ,\, \alpha$ be  real numbers. Then 
	$$e^{\phi e_{13}e^{-e_{12}\theta}}e^{-\alpha e_{12}\theta}=\cos\phi \, e^{-\alpha e_{12}\theta}+\sin\phi\,  e^{(\alpha +1)e_{12}\theta}e_{13} .$$
\end{lem}

\begin{proof} Let $\mu =e_{12}\cos\theta +e_{23}\sin\theta$. Then $\mu^2=-1$, so
	\begin{align*}
		e^{\phi e_{13}e^{-e_{12}\theta}}e^{-\alpha e_{12}\theta}&=e^{\phi [e_{13}\cos\theta+e_{23}\sin\theta]}e^{-\alpha e_{12}\theta}\\
		&=[\cos\phi +(e_{13}\cos\theta+e_{23}\sin\theta)\sin\phi ][\cos (\alpha\theta )-e_{12}\sin (\alpha\theta )]\\
		&=\cos (\alpha\theta )[\cos\phi +(e_{13}\cos\theta +e_{23}\sin\theta )\sin\phi ]\\
		&\qquad -\sin (\alpha\theta )[\cos\phi +(e_{13}\cos\theta +e_{23}\sin\theta )\sin\phi ]\, e_{12}\\
		&=\cos\phi [\cos (\alpha\theta )-e_{12}\sin (\alpha\theta )]\\
		&+[\cos (\alpha\theta )+e_{12}\sin (\alpha\theta )][e_{13}\cos\theta +e_{23}\sin\theta ]\sin\phi\\
		&=\cos\phi [\cos (\alpha\theta )-e_{12}\sin (\alpha\theta )]\\
		&+[\cos (\alpha\theta )+e_{12}\sin (\alpha\theta )][\cos\theta +e_{12}\sin\theta ]e_{13}\sin\phi\\
		&=\cos \phi \, e^{-\alpha e_{12}\theta}+e^{\alpha e_{12}\theta}e^{e_{12}\theta}e_{13}\sin\phi\\
		&=\cos\phi \, e^{-\alpha e_{12}\theta}+e^{(\alpha +1)e_{12}\theta}e_{13}\sin\phi .
	\end{align*}
\end{proof}
An application of \eqref{Gamma_Spherical_Version} and Lemma \ref{conjugation_ephi} gives
\begin{align*}
	\Gamma H_{k}^{2n}(r,\theta,\phi)
	&= \frac{1}{\sin\phi}e_{23}e^{-e_{12}\theta}e^{\phi e_{13}e^{-e_{12}\theta}}\bigg[\frac{r^{k}}{2^{n+1}\sqrt{\pi}}(\sin\phi)^{n}P_{k-n}^{(n,n)}(\cos\phi)(-n\sin n\theta)\bigg]\\
	&+e_{13}e^{-e_{12}\theta}\bigg[\frac{r^{k}}{2^{n+1}\sqrt{\pi}}n(\sin\phi)^{n-1}(\cos\phi)P_{k-n}^{(n,n)}(\cos\phi)(\cos n\theta)\\
	&+\frac{r^{k}}{2^{n+1}\sqrt{\pi}}(\sin\phi)^{n}P_{k-n}^{(n,n)'}(\cos\phi)(-\sin\phi)(\cos n\theta)\bigg]\\
	&=-ne_{23}e^{-e_{12}\theta}e^{\phi e_{13}e^{-e_{12}}\theta}\frac{r^{k}}{2^{n+1}\sqrt{\pi}}(\sin\phi)^{n-1}P_{k-n}^{(n,n)}(\cos\phi)\sin k\theta\\
	&+ne_{13}e^{-e_{12}\theta}\frac{r^{k}}{2^{n+1}\sqrt{\pi}}(\sin\phi)^{n-1}(\cos\phi)P_{k-n}^{(n,n)}(\cos\phi)\cos n\theta\\
	&-\frac{r^{k}}{2^{n+1}\sqrt{\pi}}e_{13}e^{-e_{12}\theta} (\sin\phi)^{n+1} P_{k-n}^{(n,n)'}(\cos\phi)\cos n\theta.
\end{align*}
Now by \eqref{jecobi_derivative_recurrence} with $\alpha=\beta=n$, we have 
\begin{align*}
	\Gamma H_{k}^{2n}(r,\theta,\phi)&=-ne_{23}e^{-e_{12}\theta}e^{\phi e_{13}e^{-e_{12}\theta}}\frac{r^{k}}{2^{n+1}\sqrt{\pi}}(\sin\phi)^{n-1}P_{k-n}^{(n,n)}(\cos\phi)\sin n\theta\\
	&+ne_{13}e^{-e_{12}\theta}\frac{r^{k}}{2^{n+1}\sqrt{\pi}}(\sin\phi)^{n-1}(\cos\phi)P_{k-n}^{(n,n)}(\cos\phi)\cos n\theta\\
	&-\frac{r^{k}}{2^{n+1}\sqrt{\pi}}e_{13}e^{-e_{12}\theta} (\sin\phi)^{n-1} n P_{k-n-1}^{(n,n)}(\cos\phi) \cos n\theta\\
	&+\frac{r^{k}}{2^{n+1}\sqrt{\pi}}e_{13}e^{-e_{12}\theta} (\sin\phi)^{n-1} (k-n) P_{k-n}^{(n,n)}(\cos\phi) \cos n\theta\\
	&= \frac{r^{k}}{2^{n+1}\sqrt{\pi}}(\sin\phi)^{n-1}\bigg[-n e_{23}e^{-e_{12}\theta}e^{\phi e_{13}e^{-e_{12}\theta}}P_{k-n}^{(n,n)}(\cos\phi)\sin n\theta\\
	-&k e_{13}e^{-e_{12}\theta}\, P_{k-n-1}^{(n,n)}(\cos\phi)\cos n\theta+k e_{13}e^{-e_{12}\theta} (\cos\phi)  \, P_{k-n}^{(n,n)}(\cos\phi)\cos n\theta\bigg].
\end{align*}
Similarly, 
\begin{align*}
	\Gamma H_{k}^{2n+1}(r,\theta,\phi)=& \frac{r^{k}}{2^{n+1}\sqrt{\pi}}(\sin\phi)^{n-1}\bigg[n e_{23}e^{-e_{12}\theta}e^{\phi e_{13}e^{-e_{12}\theta}} P_{k-n}^{(n,n)}(\cos\phi)\cos n\theta\\
	-&k e_{13}e^{-e_{12}\theta}\, P_{k-n-1}^{(n,n)}(\cos\phi)\sin n\theta\\
	+&k e_{13}e^{-e_{12}\theta}(\cos\phi)\, P_{k-n}^{(n,n)}(\cos\phi)\sin n\theta\bigg].
\end{align*}
Let $\omega:=\frac{x}{\vert x\vert}$. We define the the function $Y_{k-1}^{2n}:\mathbb{R}^{3}\to\mathbb{R}_{3}$ by 
$Y_{k-1}^{2n}(\theta,\phi):=\partial_{x} H_{k}^{2n}(r,\theta,\phi)\bigg\vert_{r=1}$. Then by \eqref{dirac_operator_polar_form_gamma}
\begin{align}
	-\omega & Y_{k-1}^{2n}(\theta,\phi)=-\omega\partial_{x} H_{k}^{2n}(r,\theta,\phi)\bigg\vert_{r=1}\nonumber\\
	&=\frac{(\sin\phi)^{n-1}}{2^{n+1}\sqrt{\pi}}\bigg[ kP_{k-n}^{(n,n)}(\cos\phi)\cos n\theta \, e_{13}e^{-e_{12}\theta}e^{-\phi e_{13}e^{-e_{12}\theta}}\nonumber \qquad\qquad\\
	&-n\, e_{23}e^{-e_{12}\theta}e^{\phi e_{13}e^{-e_{12}\theta}}P_{k-n}^{(n,n)}(\cos\phi)\sin n\,\theta\label{even_monogenics_after_dirac_operator} \nonumber\\
	&-k\, e_{13}e^{-e_{12}\theta}\, P_{k-n-1}^{(n,n)}(\cos\phi)\cos n\theta\bigg],
\end{align}
for $1\leq n\leq k-1$. Similarly, we define $Y_{k-1}^{2n+1}:\mathbb{R}^{3}\to\mathbb{R}_{3}$ by 
$Y_{k-1}^{2n+1}(\theta,\phi):=\partial_{x} H_{k}^{2n+1}(r,\theta,\phi)\bigg\vert_{r=1}$. Then
\begin{align}
	-\omega Y_{k-1}^{2n+1}(\theta,\phi):=&\frac{(\sin\phi)^{n-1}}{2^{n+1}\sqrt{\pi}}\bigg[ kP_{k-n}^{(n,n)}(\cos\phi)\sin n\theta\, e_{13}e^{-e_{12}\theta}e^{-\phi e_{13}e^{-e_{12}\theta}}\nonumber\\
	+&n e_{23}e^{-e_{12}\theta}e^{\phi e_{13}e^{-e_{12}\theta}}P_{k-n}^{(n,n)}(\cos\phi)\cos n\theta\nonumber\\
	-&k e_{13}e^{-e_{12}\theta}\, P_{k-n-1}^{(n,n)}(\cos\phi)\sin n\theta\bigg], \label{odd_monogenics_after_dirac_operator}
\end{align}
for $1\leq n\leq k-1$. For this range of $n$, let
\begin{equation}\label{monogenic_degree_kminusone}
	F_{k-1}^{n}(\theta,\phi):= (Y_{k-1}^{2n}-Y_{k-1}^{2n+1}e_{12}).
\end{equation} 
By equations \eqref{even_monogenics_after_dirac_operator}, \eqref{odd_monogenics_after_dirac_operator}, and \eqref{monogenic_degree_kminusone}, we have
\begin{align}
	F_{k-1}^{n}(\theta,\phi)=\omega\frac{(\sin\phi)^{n-1}}{2^{n+1}\sqrt{\pi}}\bigg[& (k-n) e_{13}e^{-e_{12}\theta}e^{-\phi e_{13}e^{-e_{12}\theta}}e^{-e_{12}n\theta}P_{k-n}^{(n,n)}(\cos\phi)\nonumber\\
	&-k e_{13}e^{-(n+1)e_{12}\theta}\, P_{k-n-1}^{(n,n)}(\cos\phi)\bigg].\label{first_nminusone_monogenic}
\end{align}
We can also calculate $F_{k-1}^{0}(\theta,\phi)$ using \eqref{zero_harmonic} as follows
\begin{align}
	F_{k-1}^{0}(\theta,\phi)&=Y_{k-1}^{0}(\theta,\phi)=\omega (\frac{\partial}{\partial r}+\frac{1}{r}\Gamma)H_{k}^{0}\bigg\vert_{r=1}\notag\\
	&=\omega\frac{k}{\sqrt{8\pi}}P_{k}^{(0,0)}(\cos\phi)-\omega\frac{1}{\sqrt{8\pi}}P_{k}^{(0,0)'}(\cos\phi)\sin\phi \, e_{13}e^{-e_{12}\theta}. \label{Definition_of_Fzero_nminusone}
\end{align}
\begin{thm}\label{Orthonormality_Fk}
	The functions $\{F_{k-1}^{n}\}_{n=0}^{k-1}$ form an orthogonal basis for $\mathcal{M}_{\ell}^{+}(k-1)$ in $\mathbb{R}^{3}$.
\end{thm}
\begin{proof}
	We first note that each $Y_{k-1}^{n}=\partial H_{k}^{n}$ is monogenic since $\partial_{x}Y_{k-1}^{n}=-\triangle H_{k}^{n}=0$. Also, since $H_{k}^{n}$ is homogeneous of degree $k,$  $Y_{k-1}^{n}$ is homogeneous of degree $k-1$. Hence $Y_{k-1}^{n}\in \mathcal{M}_{\ell}^{+}(k-1)$. Since $F_{k-1}^{n}=Y_{k-1}^{2n}-Y_{k-1}^{2n+1}e_{12}$, we see that $F_{k-1}^{n}\in \mathcal{M}_{\ell}^{+}(k-1)$ for $0\leq n \leq k$. For $1\leq n\leq k-1,$ we have by \eqref{first_nminusone_monogenic} and Lemma \ref{conjugation_ephi}
	\begin{align}
		&\langle F_{k-1}^{n}, F_{k-1}^{n'}\rangle_{L^{2}(S^{2})}=\int\limits_{S^{2}} \overline{F_{k-1}^{n}(\omega)}F_{k-1}^{n'}(\omega) d\omega\nonumber\\
		&=\int\limits_{0}^{\pi}\int\limits_{0}^{2\pi}\bigg[\frac{(\sin\phi)^{n-1}}{2^{n+1}\sqrt{\pi}}(k-n)e^{e_{12}k\theta}e^{\phi e_{13}e^{-e_{12}\theta}}e^{e_{12}\theta}\overline{e_{13}} P_{k-n}^{(n,n)}(\cos\phi)\bigg]\nonumber\\
		&\hspace*{.5cm}\times\bigg[\frac{(\sin\phi)^{n'-1}}{2^{n'+1}\sqrt{\pi}}(k-n')e_{13}e^{-e_{12}\theta}e^{-\phi e_{13}e^{-e_{12}\theta}}e^{-e_{12}k'\theta}P_{k-n'}^{(n',n')}(\cos\phi)\bigg](\sin\phi)d\theta\, d\phi\nonumber\\
		&+\int\limits_{0}^{\pi}\int\limits_{0}^{2\pi}\bigg[\frac{(\sin\phi)^{n-1}}{2^{n+1}\sqrt{\pi}}(k-n)e^{e_{12}k\theta}e^{\phi e_{13}e^{-e_{12}\theta}}e^{e_{12}\theta}\overline{e_{13}} P_{k-n}^{(n,n)}(\cos\phi)\bigg]\nonumber\\
		&\hspace*{.5cm}\times\bigg[-k\frac{(\sin\phi)^{n'-1}}{2^{n'+1}\sqrt{\pi}} e_{13}e^{-(n'+1)e_{12}\theta}\, P_{k-n'-1}^{(n',n')}(\cos\phi)\bigg](\sin\phi)d\theta\, d\phi\nonumber\\
		&+\int\limits_{0}^{\pi}\int\limits_{0}^{2\pi}\bigg[-k\frac{(\sin\phi)^{n-1}}{2^{n+1}\sqrt{\pi}} e^{(n+1)e_{12}\theta}\overline{e_{13}}\, P_{k-n-1}^{(n,n)}(\cos\phi)\bigg]\nonumber\\
		&\hspace*{.5cm}\times\bigg[ \frac{(\sin\phi)^{n'-1}}{2^{n'+1}\sqrt{\pi}}(k-n')e_{13}e^{-e_{12}\theta}e^{-\phi e_{13}e^{-e_{12}\theta}}e^{-e_{12}n'\theta}P_{k-n'}^{(n',n')}(\cos\phi)\bigg](\sin\phi)d\theta\, d\phi\nonumber\\
		&+\int\limits_{0}^{\pi}\int\limits_{0}^{2\pi}\bigg[-k\frac{(\sin\phi)^{n-1}}{2^{n+1}\sqrt{\pi}} e^{e_{12}(n+1)\theta}\overline{e_{13}}\, P_{k-n-1}^{(n,n)}(\cos\phi)\bigg]\nonumber\\
		&\hspace*{.5cm}\times\bigg[-k\frac{(\sin\phi)^{n'-1}}{2^{n'+1}\sqrt{\pi}} e_{13}e^{-e_{12}(n'+1)\theta}\, P_{k-n'-1}^{(n',n')}(\cos\phi) \bigg](\sin\phi)d\theta\, d\phi\nonumber\\
		&=I_{1}+I_{2}+I_{3}+I_{4}.
	\end{align}
	We treat $ I_{1},\, I_{2},\,I_{3},\,I_{4} $ separately. We use the fact that  $\int\limits_{0}^{2\pi}e^{e_{12}k\theta}d\theta=2\pi\delta_{k0}$. Then
	\begin{align}
		I_{1}
		&=\frac{(k-n)(k-n')}{2^{n+n'+2}\pi}\int\limits_{0}^{\pi}(\sin\phi)^{n+n'-1} P_{k-n}^{(n,n)}(\cos\phi)P_{k-n'}^{(n',n')}(\cos\phi) \, d\phi \int\limits_{0}^{2\pi} e^{e_{12}(n-n')\theta} d\theta\nonumber\\
		&= \frac{(k-n)^{2}}{2^{2n+1}}\int\limits_{0}^{\pi}(\sin\phi)^{2n-1} \big(P_{k-n}^{(n,n)}(\cos\phi)\big)^{2} \, d\phi \,\delta_{nn'}, \label{I1_orthogonality}
	\end{align}
	With an application of Lemma \ref{Commutation_Lemma}, we have
	\begin{align}
		I_{2}
		&= \frac{k(n-k)}{2^{n+n'+2}\pi}\int\limits_{0}^{\pi}\int\limits_{0}^{2\pi} (\sin\phi)^{n+n'-1} P_{k-n}^{(n,n)}(\cos\phi)P_{k-n'-1}^{(n',n')}(\cos\phi)\nonumber\\  
		&\qquad\qquad\qquad\qquad\qquad \times e^{e_{12}n\theta}e^{\phi e_{13}e^{-e_{12}\theta}} e^{-e_{12}n'\theta} d\theta\, d\phi \nonumber\\
		&=\frac{k(n-k)}{2^{n+n'+2}\pi}\int\limits_{0}^{\pi}\int\limits_{0}^{2\pi} (\sin\phi)^{n+n'-1} P_{k-n}^{(n,n)}(\cos\phi)\, P_{k-n'-1}^{(n',n')}(\cos\phi)\,  e^{e_{12}n\theta} \nonumber\\
		&\qquad\qquad\qquad\qquad\qquad\times 
		\big[   e^{-e_{12}n'\theta} \cos\phi + e^{e_{12}(n'+1)\theta} \, e_{13} \sin\phi   \big]  d\theta\, d\phi \nonumber\\
		&= \frac{k(n-k)}{2^{2n+1}}\int\limits_{0}^{\pi} (\cos \phi)(\sin\phi)^{2n-1} P_{k-n}^{(n,n)}(\cos\phi)\, P_{k-n-1}^{(n,n)}(\cos\phi)  \, d\phi\, \delta_{nn'}. \label{I2_orthogonality}
	\end{align}
Similarly,
	\begin{align}
		I_{3}
		&=\frac{k(n'-k)}{2^{n+n'+2}\pi}\int\limits_{0}^{\pi}\int\limits_{0}^{2\pi} (\sin\phi)^{n+n'-1} P_{k-n-1}^{(n,n)}(\cos\phi)P_{k-n'}^{(n',n')}(\cos\phi)  e^{e_{12}n\theta} \nonumber\\
		&\qquad\qquad\qquad\qquad\qquad\times 
		e^{-\phi e_{13}e^{-e_{12}\theta}} e^{-e_{12}n'\theta}  d\theta\, d\phi \nonumber\\
		&= \frac{k(n'-k)}{2^{2n+1}}\int\limits_{0}^{\pi} (\cos \phi) (\sin\phi)^{2n-1} P_{k-n-1}^{(n,n)}(\cos\phi)\, P_{k-n}^{(n,n)}(\cos\phi) \,  d\phi\, \delta_{nn'}, \label{I3_orthogonality}
	\end{align}
	and, 
	\begin{align}
		I_{4}&=\frac{k^{2}}{2^{n+n'+2}\pi}\int\limits_{0}^{\pi}\int\limits_{0}^{2\pi} (\sin\phi)^{n+n'-1} P_{k-n-1}^{(n,n)}(\cos\phi)\nonumber\\
		&\hspace*{3cm}P_{k-n'-1}^{(n',n')}(\cos\phi) \, e^{e_{12}(n+1)\theta}  e^{-e_{12}(n'+1)\theta}  d\theta\, d\phi \nonumber\\
		&= \frac{k^{2}}{2^{2n+1}}\int\limits_{0}^{\pi}(\sin\phi)^{2n-1} \big(P_{k-n-1}^{(n,n)}(\cos\phi)\big)^{2}\, d\phi\,\delta_{nn'}. \label{I4_orthogonality}
	\end{align}
	From \eqref{Definition_of_Fzero_nminusone}, for $1\leq n\leq k-1$ multiple applications of Lemma \ref{Commutation_Lemma} gives
	
	\begin{align}
		&\langle F_{k-1}^{0}, F_{k-1}^{n}\rangle_{L^{2}(S^{2})}=\int\limits_{0}^{\pi}\int\limits_{0}^{2\pi}\overline{\bigg[\frac{k}{\sqrt{8\pi}}P_{k}^{(0,0)}(\cos\phi)-\frac{1}{\sqrt{8\pi}}P_{k}^{(0,0)'}(\cos\phi)\sin\phi e_{13}e^{-e_{12}\theta}\bigg]} \nonumber\\
		&\hspace*{1cm}\times \bigg[\frac{(\sin\phi)^{n-1}}{2^{n+1}\sqrt{\pi}}(k-n) e_{13}e^{-e_{12}\theta}e^{-\phi e_{13}e^{-e_{12}\theta}}e^{-e_{12}n\theta}P_{k-n}^{(n,n)}(\cos\phi)\nonumber\\
		&\hspace*{3cm}-k\frac{(\sin\phi)^{n-1}}{2^{n+1}\sqrt{\pi}} e_{13}e^{-(n+1)e_{12}\theta}\, P_{k-n-1}^{(n,n)}(\cos\phi)\bigg]\, d\theta\, d\phi\nonumber\\
		&=\int\limits_{0}^{\pi}\int\limits_{0}^{2\pi}\bigg[ \frac{k(k-n)(\sin\phi)^{n-1}}{2^{n+\frac{5}{2}} \pi}P_{k}^{(0,0)}(\cos\phi) \big[e_{13}\, e^{-(n+1)\, e_{12}\theta}\, \cos\phi+e^{-n e_{12}\, \theta}\sin\phi  \big]\nonumber\\
		&\hspace{5cm}\times P_{k-n}^{(n,n)}(\cos\phi) \bigg]\, d\theta\, d\phi\nonumber\\
		&-\int\limits_{0}^{\pi}\int\limits_{0}^{2\pi}\bigg[\frac{k^{2}(\sin\phi)^{n-1}}{2^{n+\frac{5}{2}}\pi}P_{k}^{(0,0)}(\cos\phi) e_{13}e^{-(n+1)e_{12}\theta}\, P_{k-n-1}^{(n,n)}(\cos\phi) \bigg]\, d\theta\, d\phi\nonumber\\
		&+\int\limits_{0}^{\pi}\int\limits_{0}^{2\pi}\bigg[ \frac{(n-k)(\sin\phi)^{n}}{2^{n+\frac{5}{2}} \pi} P_{k}^{(0,0)'}(\cos\phi) \big[e^{-n\, e_{12}\theta}\, \cos\phi-e^{(n+1) e_{12}\, \theta}\sin\phi  \big]\nonumber\\
		&\hspace{5cm}\times P_{k-n}^{(n,n)}(\cos\phi) \bigg]\, d\theta\, d\phi\nonumber\\
		&+\int\limits_{0}^{\pi}\int\limits_{0}^{2\pi}\bigg[ \frac{k(\sin\phi)^{n}}{2^{n+\frac{5}{2}}\pi} P_{k}^{(0,0)'}(\cos\phi) e^{e_{12}n\theta} P_{k-n-1}^{(n,n)}(\cos\phi)\bigg]\, d\theta\, d\phi. \label{orthogonality_F_zero_F_n}
	\end{align}
	Again since $\int\limits_{0}^{2\pi}e^{e_{12}n\theta}d\theta=2\pi\delta_{n0}$
	all the integrals in \eqref{orthogonality_F_zero_F_n} vanish
	for $1\leq n\leq k-1$. Hence for $1\leq n\leq k-1$,
	\begin{equation}\label{F0_Fn_orthogonality}
		\langle F_{k-1}^{0}, F_{k-1}^{n}\rangle_{L^{2}(S^{2})}=0.
	\end{equation}
	By \eqref{I1_orthogonality}--\eqref{I4_orthogonality} and \eqref{F0_Fn_orthogonality}, the functions $\{F_{k-1}^{0},F_{k-1}^{1},\dots , F_{k-1}^{k-1}\}$ form an orthogonal system of spherical monogenics of degree $k-1$ in $\mathbb{R}^{3}$. Since $\dim \big(\mathcal{M}_{\ell}^{+}(k-1)\big)=k$, this collection span $ \mathcal{M}_{\ell}^{+}(k-1)$ and is therefore an orthogonal basis for $ \mathcal{M}_{\ell}^{+}(k-1)$.
\end{proof}

\begin{rem}
	By direct calculations, we have
	\begin{align}
		-\omega Y_{k-1}^{2k}(\theta,\phi):=-\omega \partial_{x} H_{k}^{2k}(r,\theta,\phi)\bigg\vert_{r=1}&=\frac{k (\sin\phi)^{k-1}}{2^{k+1}\sqrt{\pi}}e^{e_{12}\theta}e^{-\phi e_{13}e^{e_{12}\theta}}e^{-e_{12}k\theta}e_{13},\label{even_last_monogenic}\\
		-\omega Y_{k-1}^{2k+1}(\theta,\phi):=-\omega \partial_{x} H_{k}^{2k+1}(r,\theta,\phi)\bigg\vert_{r=1}&=\frac{k (\sin\phi)^{k-1}}{2^{k+1}\sqrt{\pi}}e^{e_{12}\theta}e^{-\phi e_{13}e^{e_{12}\theta}}e^{-e_{12}k\theta}e_{23}.\label{odd_last_monogenic}
	\end{align}
	We observe that $Y_{k-1}^{2k}=Y_{k-1}^{2k+1}e_{12}$. So
	\begin{equation}\label{last_monogenic}
		F_{k-1}^{k}(\theta,\phi):=\omega \frac{k (\sin\phi)^{k-1}}{2^{k+1}\sqrt{\pi}}e^{e_{12}\theta}e^{-\phi e_{13}e^{e_{12}\theta}}e^{-e_{12}k\theta}e_{13}.
	\end{equation}
	However, from \eqref{first_nminusone_monogenic},
	\begin{align}
		F_{k-1}^{k-1}(\theta,\phi)&=\omega\frac{(\sin\phi)^{k-2}}{2^{k}\sqrt{\pi}}\bigg[e_{13}e^{-e_{12}\theta}e^{-\phi e_{13}e^{-e_{12}\theta}}e^{-e_{12}(k-1)\theta} k \cos\phi -k e_{13}e^{-ke_{12}\theta}\,\bigg]\nonumber\\
		&=\omega\frac{(\sin\phi)^{k-2}k\, e_{13}}{2^{k}\sqrt{\pi}}\bigg[e^{-\phi \, e_{13}\,  e^{e_{12}\theta}} \cos\phi-1\,\bigg]e^{-e_{12}k\theta}\nonumber\\
		&=\omega\frac{(\sin\phi)^{k-2}k\, e_{13}}{2^{k}\sqrt{\pi}}\bigg[(\cos\phi -e_{13}e^{e_{12}\theta}\sin\phi) \cos\phi\nonumber\\
		&\hspace*{5cm} -(\cos\phi)^{2}-(\sin\phi)^{2}\,\bigg]e^{-e_{12}k\theta}\nonumber\\
		&=\omega\frac{-(\sin\phi)^{k-1}k\, e_{13}}{2^{k}\sqrt{\pi}}\bigg[e_{13}e^{e_{12}\theta} \cos\phi+\sin\phi\,\bigg]e^{-e_{12}k\theta}\nonumber\\
		&=\omega \frac{k(\sin\phi)^{k-1}}{2^{k}\sqrt{\pi}}e^{e_{12}\theta}e^{-\phi e_{13}e^{e_{12}\theta}}e^{-e_{12}k\theta},\label{k_minus_one_monogenic}
	\end{align}
	and we conclude from \eqref{last_monogenic} and \eqref{k_minus_one_monogenic} that $F_{k-1}^{k}=F_{k-1}^{k-1}e_{13}$.
\end{rem}

Now we normalize the orthogonal collection $\lbrace F_{k-1}^{n}\rbrace_{n=0}^{k-1}$.
\begin{thm}
Let $\psi:\mathbb{R}^{m}\to \mathbb{R}$ be a real-valued spherical harmonic function of degree $k$. Then
$$\Vert \partial_{x} \psi\Vert_{L^{2}(S^{m-1})}^{2}=k(2k+m-2)\int\limits_{S^{m-1}} \psi(\omega)^{2}\, d\omega.$$
\end{thm}
\begin{proof}
Observe that since $\psi$ is real-valued and $\frac{\partial }{\partial r}(\psi(r\omega))\big\vert_{r=1}=k\,\psi(\omega),\;\; (\omega\in S^{m-1})$, we have
\begin{align*}
\Vert \partial_{x} \psi\Vert_{L^{2}(S^{m-1})}^{2}&=\bigg[ \int\limits_{S^{m-1}} \overline{(\partial_{x}\psi(\omega))}(\partial_{x}\psi(\omega))\, d\omega \bigg]_{0}\\
&=\bigg[ \int\limits_{S^{m-1}} \overline{\bigg(\frac{\partial}{\partial r}+\frac{1}{r}\Gamma\bigg)\psi(\omega)}\bigg(\frac{\partial}{\partial r}+\frac{1}{r}\Gamma\bigg)\psi(\omega)\, d\omega \bigg]_{0}\\
&=\bigg[ \int\limits_{S^{m-1}} \overline{(k\, \psi(\omega)+\Gamma\psi(\omega))}(k\,\psi(\omega)+\Gamma\psi(\omega))\, d\omega \bigg]_{0}\\
&=\bigg[ k^{2}\int\limits_{S^{m-1}} \psi(\omega)^{2} d\omega+k\int\limits_{S^{m-1}} \overline{\psi(\omega)}(\Gamma\psi(\omega)) d\omega \\
&+k\int\limits_{S^{m-1}} (\overline{\Gamma\psi(\omega)})\psi(\omega) d\omega+\int\limits_{S^{m-1}} (\overline{\Gamma\psi(\omega)})(\Gamma\psi(x)) d\omega \bigg]_{0}.
\end{align*}
However, $\psi$ is real-valued so $\Gamma\psi(x)\in \mathbb{R}_{2}^{3}$ and $\overline{\Gamma\psi(x)}=-\Gamma\psi(x)$. Also, $\Gamma$ is self-adjoint on $L^{2}(S^{m-1})$, so
\begin{align}
\Vert \partial_{x} \psi\Vert_{L^{2}(S^{m-1})}^{2}&=\bigg[ k^{2}\int\limits_{S^{m-1}} \psi(\omega)^{2} d\omega+k\int\limits_{S^{m-1}} \psi(\omega) \Gamma \psi(\omega) d\omega\nonumber\\ &-k\int\limits_{S^{m-1}} \psi(\omega) (\Gamma \psi(\omega)) d\omega+\int\limits_{S^{m-1}} (\overline{\Gamma\psi(\omega)})(\Gamma\psi(\omega)) d\omega \bigg]_{0}\nonumber\\
&=\bigg[ k^{2}\int\limits_{S^{m-1}} \psi(\omega)^{2} d\omega-\int\limits_{S^{m-1}} \psi(\omega)(\Gamma^{2}\psi(\omega)) d\omega \bigg]_{0}. \label{first_eq_norm_monogenic}
\end{align}
With an application of \eqref{laplace_Beltrami_eigenvalue} we have
\begin{align}
\int\limits_{S^{m-1}} \psi(\omega)(\Gamma^{2}\psi(\omega)) d\omega&=\int\limits_{S^{m-1}}\overline{\psi(\omega)} [ (-\triangle _{T}+(m-2)\Gamma)\psi(\omega) ] d\omega\nonumber\\
&=k(k+m-2)\int\limits_{S^{m-1}}\psi(\omega)^{2} d\omega\nonumber\\
&+(m-2) \int\limits_{S^{m-1}} \psi(\omega) (\Gamma \psi(\omega)) d\omega, \label{second_eq_norm_monogenic}
\end{align}
and by using \eqref{dirac_operator_polar_form_gamma} and the Clifford-Stokes Theorem \ref{Clifford-Stokes theorem} we have 
\begin{align}
\int\limits_{S^{m-1}} \psi(\omega) (\Gamma \psi(\omega)) d\omega&=\int\limits_{S^{m-1}} \psi(\omega) (-\omega \partial_{x} \psi(\omega)-k\psi(\omega)) d\omega\nonumber\\
&= -k\int\limits_{S^{m-1}} \psi(\omega)^{2} d\omega -\int\limits_{B(1)} (\psi(\omega)\partial_{x})(\partial_{x} \psi(\omega)) d\sigma(x)\nonumber\\
&\qquad\qquad-\int\limits_{B(1)} \psi(\omega)(\partial_{x}^{2} \psi(\omega)) d\omega\nonumber\\
&=-k\int\limits_{S^{m-1}} \psi(\omega)^{2} d\omega -\int\limits_{B(1)} (\psi(\omega)\partial_{x})(\partial_{x} \psi(\omega)) d\omega.\label{third_eq_norm_monogenic}
\end{align}
Therefore, by \eqref{first_eq_norm_monogenic}, \eqref{second_eq_norm_monogenic}, and \eqref{third_eq_norm_monogenic} and the fact that $\partial_{x}\psi$ is homogeneous of degree $k-1$, we have that
\begin{align}
&\Vert \partial_{x} \psi\Vert_{L^{2}(S^{m-1})}^{2}=\bigg[ k(2k+m-2)\int\limits_{S^{m-1}} \psi(\omega)^{2} d\omega-k(m-2)\int\limits_{S^{m-1}} \psi(\omega)^{2} d\omega\nonumber\\
&\qquad\qquad\qquad +(m-2)\int\limits_{B(1)} \overline{(\partial_{x}\psi(\omega))}(\partial_{x}\psi(\omega)) dx \bigg]_{0} \nonumber\\
&=\bigg[2k^{2}\int\limits_{S^{m-1}} \psi(\omega)^{2} d\omega+(m-2)\int\limits_{0}^{1}\int\limits_{S^{m-1}}\overline{(\partial_{x}\psi)}(\omega)(\partial_{x}\psi)(\omega) r^{2k+m-3}\, dr\, d \omega\bigg]_{0}\nonumber\\
&=\bigg[2k^{2}\int\limits_{S^{m-1}} \psi(\omega)^{2} d\omega+\frac{(m-2)}{(2k-2+m)}\int\limits_{S^{m-1}}\overline{(\partial_{x}\psi(\omega))}(\partial_{x}\psi(\omega))\, d\omega\bigg]_{0} \nonumber\\
&=2k^{2}\Vert \psi\Vert_{L^{2}(S^{m-1})}^{2}+\frac{m-2}{(2k-2+m)}\Vert \partial_{x}\psi\Vert_{L^{2}(S^{m-1})}^{2}
\label{integral_delPsi_bar_del_Psi}
\end{align}
where we have used the homogeneity of $\partial_{x} \psi$. Rearranging \eqref{integral_delPsi_bar_del_Psi} yields
$$\Vert \partial_{x} \psi\Vert_{L^{2}(S^{m-1})}^{2}=k(2k+m-2)\int\limits_{S^{m-1}} \psi(\omega)^{2}\, d\omega$$
as desired.
\end{proof}
\begin{cor}
The collection $ \bigg\lbrace \bar{F}_{k-1}^{n}:=\dfrac{F_{k-1}^{n}}{\sqrt{k(2k+1)}}\bigg\rbrace_{n=0}^{k-1} $ constitutes an orthonormal basis for $\mathcal{M}_{\ell}^{+}(k-1)$ in $\mathbb{R}^{3}$.
\end{cor}

\begin{rem}
As mentioned earlier in this chapter, a similar idea has been applied in \cite{caccao2004complete} and \cite{bock2010generalized}. However, there are some differences that we mention here.
\begin{itemize}
\item The definition of Dirac operator in \cite{caccao2004complete} and \cite{bock2010generalized}  is $\partial_{x}=\frac{\partial}{\partial x_{0}}+e_{1}\frac{\partial}{\partial x_{1}}+e_{2}\frac{\partial}{\partial x_{2}}$, and associated with the Clifford algebra $\mathbb{R}_{2,1}$.
\item The spherical monogenics in \cite{caccao2004complete} are at most quaternion-valued while the spherical monogenics developed in this section take values in $\mathbb{R}_{3}$. 
\end{itemize}
\end{rem}

\begin{rem}
In \cite{lavivcka2012complete}, the authors construct an orthonormal basis of spherical monogenics in an inductive manner for every dimension, demonstrating that it is possible to obtain spherical monogenics in a higher dimension using spherical monogenics from a lower dimension. To compute the $k$-homogeneous spherical monogenics in dimension $m$, they require having all $j$-homogeneous spherical monogenics in dimension $m-1$, where $j=0,\cdots, k$. As a result, for the case of dimension $3$, the representation of the spherical monogenics introduced in \eqref{monogenic_degree_kminusone} differs from the one presented in \cite{lavivcka2012complete}. In \cite{lavivcka2012complete}, the representation is as follows:
$$ f_{k,j}(x)=\big[ \frac{k! (2j)!}{(k+j)! j!} C_{k-j}^{j+\frac{1}{2}}(x_{3})+\frac{k! (2j+1)!}{(k+j+1)! j!} C_{k-j-1}^{j+\frac{3}{2}}(x_{3})\, x\, e_{3} \big](x_1 - e_{12}x_2)^{j}, $$
where $C_{k}^{\nu}(t)$ is the Gegenbauer polynomial defined on the line.
\end{rem}

\section{Constructing $m$-Dimensional Monogenic Functions by Optimization Methods}
In this section we construct spherical monogenics using reproducing kernels and optimization.
\begin{defn}
	A zonal function $g:S^{m-1}\to\mathbb{C}$ is a function of the form 
	$$ g(x)= f(\langle x,y \rangle ), $$
	for some $ f: \mathbb{R}\to \mathbb{C} $ and fixed $ y\in S^{m-1}. $
\end{defn}
\begin{defn}
	Fix $y\in S^{m-1}$ and $c\in \mathbb{C}$. The zonal spherical harmonic $Z_{y}$ is defined by
	\begin{equation}
		Z_{y}(x):=c R_{k}^{m}(x,y)
	\end{equation}
	where $R_{k}^{m}(x,y)$ is the reproducing kernel given in \eqref{Repruducing_Kernel_of_harmonic_Functions}.
\end{defn}
The collection $\{R_{k}^{m}(\cdot,\eta_{i})\}_{i=1}^{h_{m}^{k}}$ forms a basis for $\mathcal{H}_{k}^{m}$ for almost any choice of points $\eta_{1},\cdots, \eta_{h_{m}^{k}}\in S^{m-1}$ \cite{stein1971g}. With the aim of obtaining an orthonormal basis for $\mathcal{H}_{k}^{3},$ we define
\begin{equation}\label{Goal_basis_with_zonal_Harmonic}
	Z_{t}(x):=\sum_{j=1}^{2k+1}R_{k}^{m}(x,\eta_{j})a_{jt}
\end{equation}
with $t=1,\dots , 2k+1,$ and, $a_{jt}\in\mathbb{C}$. Let $G$ be the $(2k+1)\times (2k+1)$ matrix with $(i,j)-th$ entry $ [G]_{ij}:=R_{k}^{m}(\eta_{i},\eta_{j}) $. We aim to choose $\{a_{jt}\}$ so that $\{Z_{t}\}_{t=1}^{2k+1}$ becomes an orthonormal basis for $L^{2}(S^{2})$. By Lemma the reproducing kernel of $R_{k}$, we have
\begin{align}
	\int\limits_{S^{2}}\overline{Z_{t}(\omega)}Z_{t'}(\omega)d\omega\nonumber=&\int\limits_{S^{2}}\overline{\sum_{j=1}^{2k+1}R_{k}^{m}(\omega,\eta_{j})\,a_{jt}}\sum_{j'=1}^{2k+1}R_{k}^{m}(\omega,\eta_{j'})a_{j't'}d\omega\nonumber\\
	=&\sum_{j=1}^{2k+1} \sum_{j'=1}^{2k+1}  a_{jt}  \int\limits_{S^{2}}\overline{R_{k}^{m}(\omega,\eta_{j})}\, R_{k}^{m}(\omega,\eta_{j'})\, d\omega \, a_{j't'}\\
	=&\sum_{j,j'=1}^{2k+1}\overline{a_{jt}}\,G_{jj'}\,a_{j't'}=(A^{\ast}GA)_{tt'}.\label{harmonic_version_matrix}
\end{align}
Hence, we want to choose $A=(a_{jt})_{j,t=1}^{2k+1}$ so that $A^{\ast}GA=I_{2k+1}$, the $(2k+1)\times (2k+1)$ identity matrix. Note that $G^{\ast}=G$.  Let $a=[a_{i}]\in\mathbb{C}^{2k+1}$ be a column vector. By the reproducing kernel of $R_{k}$, we have
\begin{align}
	a^{\ast}\,G\,a &=\sum_{i,j=1}^{2k+1}\overline{a_{i}} (G_{ij}) a_{j}\nonumber\\
&= \sum_{i,j=1}^{2k+1}\overline{a_{i}}\, R_{k}^{m}( \eta_{i},\eta_{j} )\, a_{j}\nonumber\\
&= \sum_{i,j=1}^{2k+1}\overline{a_{i}} \int\limits_{S^{m-1}}\overline{R_{k}^{m}(\omega ,\eta_{i} )}\, R_{k}^{m}(\omega, \eta_{j} ) d\omega \, a_{j}\nonumber\\
&=\int\limits_{S^{m-1}}\bigg(\overline{\sum_{i=1}^{2k+1}R_{k}^{m}( \omega ,\eta_{i})a_{i}}\bigg)\bigg(\sum_{j=1}^{2k+1}R_{k}^{m}(\omega ,\eta_{j})a_{j}\bigg) d\omega \nonumber\\
&=\int\limits_{S^{m-1}} \bigg\vert \sum_{i=1}^{2k+1}R_{k}^{m}(\omega,\eta_{i})a_{i}\bigg\vert^{2} d\omega \geq 0, \label{semidefinite_proof}
\end{align}
so that $G$ is positive semidefinite. We have $G_{ij}=(2k+1)P_{k}(\langle\eta_{i},\eta_{j}\rangle)$, where $\eta_{i}$ and $\eta_{j}$ are points on the unit sphere $S^{2}$. For almost every choice of $\{ \eta_{i}\}_{i=1}^{2k+1}$ in $S^{2}$, the determinant of $G$ is nonzero (for the proof see \cite{caron2005zero}). In this case, the eigenvalues of $G$ are all strictly positive. Let $A=G^{-\frac{1}{2}}$. If $G$ is diagonal then $A$ is diagonal and the functions $Z_{t}$ in \eqref{Goal_basis_with_zonal_Harmonic} are zonal. We aim to choose $\{ \eta_{i} \}_{i=1}^{2k+1}$ so that the matrix $G$ is diagonally dominant so that $A=G^{-\frac{1}{2}}$ will be diagonally dominant and the functions $Z_{t}$ of \eqref{Goal_basis_with_zonal_Harmonic} will be almost zonal. To achieve this, we seek to minimize the off-diagonal entries of $G$. Since $G$ is self-adjoint, it is enough to minimize 
\begin{equation}\label{Objective_Function_for_Spherical_Harmonics}
	F(\eta_{1},\dots, \eta_{2k+1})=\sum_{i<j} P_{k}(\langle\eta_{i},\eta_{j}\rangle)^2.
\end{equation}
The process is as follows:

\begin{enumerate}
	\item We start by randomly choosing $2k+1$ points on the sphere.
	\item We then move these points using the method of projected gradient descent on the sphere. In this method, we select a point $\eta_{l}$ to change such that it minimizes the objective function. 
\end{enumerate}
Given $\eta_{l} \in S^{2}$ and a vector $\vec{w}\in S^{2}$ with $\langle \eta_{l},\vec{w} \rangle=0$, let $\eta_{l}(t)=(\cos t )\,\eta_{l} + (\sin t) \,\vec{w}$. Then $\eta_{l}(t)\in S^{2}$ for all $t\in \mathbb{R}$, and $\eta_{l}(0)=\eta_{l}$. Define
$$G_{l}(t)= F(\eta_{1},\dots, \eta_{l}(t),\dots,  \eta_{2k+1}) $$ 
and minimize $G_{l}(t)$ in terms of $t$. The vector $\vec{w}$ should be chosen in a direction of steepest descent  of $G_{l}$ at the point $\eta_{l}$ on the sphere $S^{2}$. The optimization process involves iteratively updating the positions of the points until convergence is achieved. To find the tangent vector $\vec{w}$ of steepest descent, we proceed as follows. By \eqref{Objective_Function_for_Spherical_Harmonics} we have

\begin{align}
	G_{l}(t)&=F_{l}(\eta_{1},\dots,\eta_{l}(t),\dots,\eta_{2k+1})\nonumber\\
	&=\sum_{i<j, \, i,{j\neq l}}  P_{k}(\langle \eta_{i},\eta_{j}\rangle)^2+\sum_{i\neq l}  P_{k}(\langle \eta_{i}, (\cos t) \,\eta_{l}+ (\sin t) \,\vec{w} \rangle)^{2}.
\end{align}
Differentiating with respect to $t$ gives 
\begin{align*}
	G'_{l}(t)&=2\sum_{i\neq l} \, P_{k}(\langle \eta_{i},(\cos t) \,\eta_{l}+ (\sin t) \,\vec{w}\rangle)\, P'_{k}(\langle \eta_{i},(\cos t) \,\eta_{l}+ (\sin t) \,\vec{w}\rangle)\\
	&\qquad\qquad \times\langle \eta_{i},-(\sin t) \,\eta_{l}+ (\cos t) \,\vec{w}\rangle.
\end{align*}
So that 
\begin{align}
	G'_{l}(0) &= 2\sum_{i\neq l} \, P_{k}(\langle \eta_{i},\eta_{l} \rangle)\, P'_{k}(\langle \eta_{i}, \eta_{l}\rangle)  \langle \eta_{i}, \vec{w}\rangle\nonumber\\
	&=2\sum_{i\neq l} \, P_{k}(\langle \eta_{i},\eta_{l} \rangle)\, P'_{k}(\langle \eta_{i}, \eta_{l}\rangle)  \langle \eta_{i}-\langle \eta_{i},\eta_{l}\rangle\eta_{l}      , \vec{w}\rangle
\end{align}
since $\langle \eta_{l},\vec{w}\rangle=0$. Note that the vectors $\eta_{i}-\langle \eta_{i},\eta_{l}\rangle \eta_{l}$ lie in the tangent plane to the sphere at the point $\eta_{l}$. We choose $\vec{w}$ in this tangent plane so that $G'_{l}(0)$ takes the largest negative value possible. Provided 
$$\sum_{i\neq l} \, P_{k}(\langle \eta_{i},\eta_{l} \rangle)\, P'_{k}(\langle \eta_{i}, \eta_{l}\rangle)  ( \eta_{i}-\langle \eta_{i},\eta_{l}\rangle\eta_{l}     )\neq 0,$$
then the direction of $\vec{w}$ of steepest descent of $G_{l}$ is 

\begin{equation}
	\vec{w}=-\frac{\sum_{i\neq l} \, P_{k}(\langle \eta_{i},\eta_{l} \rangle)\, P'_{k}(\langle \eta_{i}, \eta_{l}\rangle)  ( \eta_{i}-\langle \eta_{i},\eta_{l}\rangle\eta_{l}     )}{\Vert \sum_{i\neq l} \, P_{k}(\langle \eta_{i},\eta_{l} \rangle)\, P'_{k}(\langle \eta_{i}, \eta_{l}\rangle)  ( \eta_{i}-\langle \eta_{i},\eta_{l}\rangle\eta_{l}     )\Vert}.
\end{equation}

\begin{ex}
	For $k=2$ we obtain $5$ spherical harmonics $Z_{1}(x)$, $Z_{2}(x)$, $Z_{3}(x)$, $Z_{4}(x)$, and $Z_{5}(x)$. By using the optimization method above we find that the collection of points at which a local minimum of the objective function occurs are 
		\begin{align}
		\eta_{1}&=(-0.9578  ,  0.1971  ,  0.2092),\nonumber\\
		\eta_{2}&=(0.5136   -0.7161    0.4726),\nonumber\\
		\eta_{3}&=(0.2730   -0.7662   -0.5817).\nonumber\\
		\eta_{4}&=(-0.6364   -0.2018   -0.7445).\nonumber\\
		\eta_{5}&=(0.2471    0.1207   -0.9614).\nonumber
	\end{align}
The objective functions of those point takes the value $F(\eta_1,\eta_2,\eta_3,\eta_4,\eta_5)=0.3209$. As a result we can calculate $A=G^{-\frac{1}{2}}$. So we have the functions $Z_{t}$ as in \eqref{Goal_basis_with_zonal_Harmonic} where
		$$A=\begin{bmatrix}
		0.4830  &  0.0473  &  0.0473  &  0.0786  &  0.0786 \\
		0.0473  &  0.4830  &  0.0786  &  0.0473  &  0.0786 \\
		0.0473  &  0.0786  &  0.4830  &  0.0786  &  0.0473 \\
		0.0786  &  0.0473  &  0.0786  &  0.4830  &  0.0473 \\
		0.0786  &  0.0786  &  0.0473  &  0.0473  &  0.4830
	\end{bmatrix}$$
Note that $A$ is diagonally dominant, so that $Z_{1}(x), \dots, Z_{5}(x)$. Codes for the computations of this example can be found at
\newline
 \href{https://github.com/HamedBaghal/Spherical\_Harmonics.git}{github.com/HamedBaghal/Spherical\_Harmonics.git}.
\end{ex}


Now we want to construct a three dimensional spherical monogenic basis in similar way. 

\begin{defn}
	Fix $y\in S^{m-1}$ and $c\in \mathbb{C}$. The zonal spherical monogenic $Z_{y}$ is defined by
	\begin{equation}
		Z_{y}(x):=cK_{k}^{m}(x,y),
	\end{equation}
	where $K_{k}^{m}(x,y)$ is the reproducing kernel given in \eqref{Repruducing_Kernel_of_Monogenic_Functions}.
\end{defn}
Given such a collection, define the functions $Z_{t}(x)$ by follows:
\begin{equation}\label{Goal basis with zonal}
	Z_{t}(x):=\sum_{j=1}^{k+1}K_{k}^{m}(x,\eta_{j})a_{jt},
\end{equation}
for $t=1,\dots ,k+1$. Here $a_{jt}\in (\mathbb{R}_{3}^{0}\oplus \mathbb{R}_{3}^{2})$ will be chosen so that the collection $\{Z_{t}\}_{t=1}^{k+1}$ forms an orthonormal basis for $\mathcal{M}_{\ell}^{+}(k)$. Following similar calculations as in \eqref{harmonic_version_matrix}, our goal is to obtain a matrix $A$ for which $A^{\ast}G A=I_{(k+1)}$, where $G$ is a $(k+1)\times (k+1)$ matrix with entries given by:
\begin{equation}\label{matrix_G_for_finding_coefficients}
	G_{ij}=K_{k}^{m}(\eta_{i},\eta_{j})=(k+1)\, C_{k}^{\frac{1}{2}}(\langle \eta_{i}, \eta_{j} \rangle)+(\eta_{i}\wedge \eta_{j})\, C_{k-1}^{\frac{3}{2}}(\langle \eta_{i}, \eta_{j} \rangle).
\end{equation}
Note that $G$ is self-adjoint since $\langle\eta_{i},\eta_{j}\rangle=\overline{\langle\eta_{j},\eta_{i}\rangle}$ and $ \overline{x\wedge y} = -x\wedge y = y\wedge x $. Additionally, the matrix $G$ is positive semidefinite. To see this, let $a\in M_{k+1,1}(\mathbb{R}_{3}^{+})$. Then the $\{Y_{k}^{\ell}\}_{\ell=1}^{k+1}$ any orthonormal basis for $\mathcal{M}_{\ell}^{+}(k)$, we have
	\begin{align}
		\langle Ga , a \rangle &=\sum_{i=1}^{k+1} \overline{\sum_{j=1}^{k+1} G_{ij}\, a_{j}}\, a_{i}\nonumber\\
		&=\sum_{i=1}^{k+1} \sum_{j=1}^{k+1} (\overline{a_{j}}\,\overline{G_{ij}})\, a_{i}\nonumber\\
		&=\sum_{i,j=1}^{k+1}\overline{a_{j}} (\sum_{l=1}^{k+1}\overline{Y_{k}^{l}(\eta_{j})}Y_{k}^{l}(\eta_{i})) a_{j}\nonumber\\
		&=\sum_{l=1}^{k+1}\overline{\sum_{j=1}^{k+1}Y_{k}^{l}(\eta_{j})\, a_{j}} \sum_{i=1}^{k+1} Y_{k}^{l}(\eta_{i})\, a_{i}=\sum_{l=1}^{k+1}\vert \sum_{l=1}^{k+1} Y_{k}^{\ell}(\eta_{i})a_{i} \vert^{2} \geq 0 \label{Positive_Definiteness_First_Relation}
	\end{align}
since $\sum_{i}Y_{k}^{\ell}(\eta_{i}) a_{i} \in \mathbb{R}_{3}$. Hence, $G$ is positive semidefinite.

By direct computation, we have the followings: 
\begin{lem}\label{Lemma_properties_star_Isomorphism}
	The map $\tau: \mathbb{R}_{3}^{+}\to \mathbb{R}_{2}$ given by 
	$$\tau(x_{0}+x_{12}e_{12}+x_{13}e_{13}+x_{23}e_{23})=x_{0}+x_{12}e_{12}+x_{13}e_{2}-x_{23}e_{1} $$
	is a $\ast$-isomorphism, i.e., 
	\begin{itemize}
		\item (i) $\tau (xy)=\tau(x)\tau(y)$ for all $x,y\in \mathbb{R}_{3}^{+}$,
		\item(ii) $\tau(\lambda x +\mu y)=\lambda \tau(x)+\mu\tau(y)$ for all $x,y \in \mathbb{R}_{3}^{+}$ and $\lambda, \mu \in \mathbb{R}$,
		\item(iii) $\tau(\overline{x})$=$\overline{\tau(x)}$ for all $x\in \mathbb{R}_{3}^{+}$.
	\end{itemize}
\end{lem}
Let $B$ be a $k\times \ell$ matrix with entries in $\mathbb{R}_{3}^{+}$. We define $\tau(B)$ to be the $k+1$ matrix with quaternionic entries given by
$$[\tau(B)]_{ij}=\tau(B_{ij}).$$
\begin{prop}
	Let $\{\eta_{i}\}_{i=1}^{k+1}\subset S^{2}$ be chosen so that the $(k+1)\times (k+1)$ matrix G as above is positive definite. Then $\tau(G)$ is positive definite.
\end{prop}
\begin{proof}
	Let $(\mathbb{R}_{2})^{k+1}$ be a quaternionic  column vector and $\tilde{y}=\tau^{-1}y\in (\mathbb{R}_{3}^{+})^{k+1}$ be the column vector with entries in $\mathbb{R}_{3}^{+}$ given by $(\tau^{-1}y)_{j}=\tau^{-1}y_{j}$. Then 
	
\begin{align}
		\sum_{i=1}^{k+1} \overline{((\tau G)\, y)_{i}}\, y_{i}&=\sum_{i=1}^{k+1} \sum_{j=1}^{k+1} \overline{y_{j}}\overline{(\tau G)_{ij}}\, y_{i}\nonumber\\
		&=\sum_{i=1}^{k+1} \sum_{j=1}^{k+1} \overline{\tau (\tilde{y})}_{j}\, \overline{(\tau (G))}_{ij}\, (\tau\tilde{y})_{i}\nonumber\\
		&=\sum_{i=1}^{k+1} \sum_{j=1}^{k+1} \overline{\tau (\tilde{y})}_{j}\, \overline{\tau G}_{ij}\, \tau (\tilde{y})_{i}\nonumber\\
		&=\tau\bigg(\sum_{i=1}^{k+1} \sum_{j=1}^{k+1} \overline{ \tilde{y}_{j}}\, \overline{G}_{ij}\, \tilde{y}_{i}\bigg)\nonumber\\
		&=\tau\bigg(\sum_{i=1}^{k+1} \sum_{j=1}^{k+1} \overline{ \tilde{y}_{j}}\, G_{ij}\, \tilde{y}_{i}\bigg)> 0, \label{Positive_Definiteness_Relation_Tau}
\end{align}
	since $G$ is positive definite.
\end{proof}
\begin{thm}
	Let $\{ \eta_{1}, \dots , \eta_{k+1} \}\subset S^{2}$ be chosen so that the $(k+1)\times (k+1)$ matrix $G$ defined in \eqref{matrix_G_for_finding_coefficients} is positive definite. Let $A$ be the $(k+1)\times (k+1)$ matrix with entries from $\mathbb{R}_{3}^{+}$ given by
	\begin{equation}\label{Coefficient Matrix of Zonal}
			A=\tau^{-1}\chi^{-1}\big(\chi\tau G\big)^{-\frac{1}{2}},
		\end{equation}
	where $\chi$ is the complex adjoint matrix defined in Definition \ref{Definition_Complex_Adjoint}. Then 
	$$A^{\ast}GA=I_{k+1}.$$
\end{thm}
\begin{proof}
	Since $G$ is positive definite on $\mathbb{R}_3^{+}$ by Lemma \ref{Lemma_properties_star_Isomorphism} we have  that $\tau(G)$ is positive definite on $\mathbb{R}_{2}^{k+1}$ and by Lemma \ref{selfadjoint by complex adjoint transformation} $\chi \tau (G)$ is positive definite on $\mathbb{R}^{2k+2}$. Hence $(\chi \tau (G))^{-\frac{1}{2}}$ exists and
\begin{align*}
	A^{\ast} G A&=[\tau^{-1}\chi^{-1}\big(\chi\tau (G)\big)^{-\frac{1}{2}}]^{\ast}\,G\, \tau^{-1}\chi^{-1}\big(\chi\tau (G)\big)^{-\frac{1}{2}}	\\
	&=\tau^{-1} \chi^{-1} \big( (\chi \tau (G))^{-\frac{1}{2}}  \chi \tau (G) \, \chi \tau (G)^{-\frac{1}{2}}  \big)\\
	&=\tau^{-1} \chi^{-1}(I_{2k+2})=\tau^{-1} I_{k+1}=I_{k+1}.
\end{align*}
\end{proof}
In the following we prove that the matrix $G$ defined in \eqref{matrix_G_for_finding_coefficients} is almost for every choice of points is positive definite. We employ spherical coordinates on $S^{2}$, i.e., we want each $\omega\in S^{2}$ as $\omega=(\cos\theta \sin\phi, \sin\theta \sin\phi, \cos\phi)$ with $0\leq \theta < 2\pi$, $0\leq \phi <\pi$. If $F:S^2 \to \mathbb{C}$ we write $F(\theta,\phi)=F(\omega)$.

\begin{lem}
	For $0\leq n\leq k$, let $F_{k}^{n}$ be as in \eqref{first_nminusone_monogenic}. Then
	\begin{align*}
		& \sum_{j=0}^{k}\overline{F_k^n(\frac{j}{k+1},\phi )}F_k^\ell (\frac{j}{k+1},\phi )=\frac{(\sin\phi )^{2n-2}}{2^{2n+2}\pi}(k+1)\big[(k+1-n)^2P_{k+1-n}^{(n,n)}(\cos\phi )^2\\
		& +(k+1)^2P_{k-n}^{(n,n)}(\cos\phi )^2 -2(k+1) (k+1-n)P_{k+1-n}^{(n,n)}(\cos\phi )P_{k-n}^{(n,n)}(\cos\phi )\cos\phi \big]\delta_{n\ell}.
	\end{align*}
\end{lem}

\begin{proof}
	Note that 
	\begin{align}
		&\sum_{n=0}^{k}\overline{F_k^n(2\pi\frac{j}{k+1},\phi )}F_k^\ell (2\pi\frac{j}{k+1},\phi )=\frac{(\sin\phi )^{n+\ell -2}}{2^{n+\ell -2}\pi}\notag\\
		&\quad\times\sum_{j=0}^{k}[(k+1-n)e^{e_{12}n\frac{j}{k+1}}e^{\phi e_{13}e^{-e_{12}\frac{j}{k+1}}}e^{e_{12}\frac{j}{k+1}}P_{k+1-n}^{(n,n)}(\cos\phi )\notag\\
		&\quad-(k+1)e^{e_{12}(n+1)\frac{j}{k+1}}P_{k+1-n}^{(n,n)}(\cos\phi )]\overline{e_{13}}\,\overline\omega\notag\\
		&\quad\times\omega e_{13}[(k+1-\ell )e^{-e_{12}\frac{j}{k+1}}e^{-\phi e_{13}e^{-e_{12}\frac{j}{k+1}}}e^{-e_{12}\ell \frac{j}{k+1}}P_{k+1-\ell}^{(\ell ,\ell )}(\cos\phi )\notag\\
		&\quad-(k+1)e^{-e_{12}(\ell +1)\frac{j}{k+1}}P_{k-\ell}^{(\ell ,\ell )}(\cos\phi )]\notag\\
		&\quad=S_1+S_2+S_3+S_4\label{S list}
	\end{align}
	with $S_1, S_2, S_3 , S_4$ to be computed separately. We have 
	\begin{align}
		S_1&=\frac{(\sin\phi )^{N+\ell -2}}{2^{n+\ell +2}\pi}P_{k+1-n}^{(n,n)}(\cos\phi )P_{k_1-\ell}^{(\ell ,\ell )}(\cos\phi )\notag\\
		&\quad\times\sum_{j=0}^{k}[(k+1-n)e^{2\pi e_{12}n\frac{j}{k+1}}e^{\phi e_{13}e^{-2\pi e_{12}\frac{j}{k+1}}}e^{2\pi e_{12}\frac{j}{k+1}}]\notag\\
		&\quad\times[(k+1-\ell )e^{-2\pi e_{12}\frac{j}{k+1}}e^{-\phi e_{13}e^{-2\pi e_{12}\frac{j}{k+1}}}e^{-2\pi e_{12}\ell \frac{j}{k+1}}]\notag\\
		&=\frac{(\sin\phi )^{N+\ell -2}}{2^{n+\ell +2}\pi}P_{k+1-n}^{(n,n)}(\cos\phi )P_{k_1-\ell}^{(\ell ,\ell )}(\cos\phi )(k+1-n)(k+1-\ell )\notag\\
		&\quad \times\sum_{j=0}^{k}e^{2\pi e_{12}(n-\ell )\frac{j}{k+1}}\notag\\
		&=\frac{(\sin\phi )^{2n-2}}{2^{2n+2}\pi}(k+1)(k+1-n)^2P_{k+1-n}^{(n,n)}(\cos\phi )^2\delta_{n\ell }.\label{S1}
	\end{align}
	A similar computation gives
	\begin{equation}
		S_4=\frac{(\sin\phi )^{2n-2}}{2^{2n+2}\pi}(k+1)^3P_{k-n}^{(n,n)}(\cos\phi )^2\delta_{n\ell}.\label{S2}
	\end{equation}
	For $S_2$ we have
	\begin{align}
		S_2&=-\frac{(\sin\phi )^{n+\ell -2}}{2^{n+\ell +2}\pi}(k+1-n)(k+1)P_{k+1-n}^{(n,n)}(\cos\phi )P_{k-\ell}^{(\ell,\ell )}(\cos\phi )\notag\\
		&\quad\times\sum_{j=0}^{k}e^{2\pi e_{12}n\frac{j}{k+1}}e^{\phi e_{13}e^{-2\pi e_{12}\frac{j}{k+1}}}e^{-2\pi e_{12}\ell \frac{j}{k+1}}\notag\\
		&=-\frac{(\sin\phi )^{n+\ell -2}}{2^{n+\ell +2}\pi}(k+1-n)(k+1)P_{k+1-n}^{(n,n)}(\cos\phi )P_{k-\ell}^{(\ell,\ell )}(\cos\phi )\notag\\
		&\quad\times \sum_{j=0}^{k}e^{2\pi e_{12}n\frac{j}{k+1}}[\cos\phi e^{-2\pi e_{12}\ell \frac{j}{k+1}}+\sin\phi e^{2\pi e_{12}(\ell +1)\frac{j}{k+1}}e_{13}]\notag\\
		&=-\frac{(\sin\phi )^{2n-2}}{2^{2n+2}\pi}(k+1)^2(k+1-n)P_{k+1-n}^{(n,n)}(\cos\phi )P_{k-n}^{(n,n)}(\cos\phi )\cos\phi\delta_{n\ell} \label{S3}
	\end{align}
	and in a similar fashion we have $S_3=S_2$. Combining (\ref{S list})-(\ref{S3}) gives the result.
\end{proof}
\begin{cor}
	Let $\{F_k^n\}_{n=0}^{k}$ be as above. Then 
	\begin{align*}
	\sum_{n=0}^{k}\overline{F_k^n(\frac{2\pi j}{k+1},\frac{\pi}{2} )}F_k^\ell (\frac{2\pi j}{k+1},\frac{\pi}{2} )&=\frac{k+1}{2^{2n+2}\pi}[(k+1-n)^2P_{k+1-n}^{(n,n)}(0)^2\\
	&\quad+(k+1)^2P_{k-n}^{(n,n)}(0 )^2]\delta_{n\ell}.		
	\end{align*}
\end{cor}

The normalised spherical monogenics $\tilde F_k^n$ are defined by $\tilde F_k^n=\frac{F_k^n}{\sqrt{(k+1)(2k+3)}}$. We therefore have
	\begin{align*}
\sum_{n=0}^{k}\overline{\tilde F_k^n(\frac{2\pi j}{k+1},\frac{\pi}{2} )}\tilde F_k^\ell (\frac{2\pi j}{k+1},\frac{\pi}{2} )&=\frac{1}{2^{2n+2}(2k+3)\pi}[(k+1-n)^2P_{k+1-n}^{(n,n)}(0)^2\\
&\quad+(k+1)^2P_{k-n}^{(n,n)}(0 )^2]\delta_{n\ell}.
	\end{align*}

\begin{thm} For almost every choice of $\eta\in (S^2)^{k+1}$, $G=G(\eta)$ is invertible.
\end{thm}

\begin{proof} Suppose $G=G(\eta)$ is singular on a set of positive measure $E\subset (S^2)^{k+1}$ with $\eta=(\eta_1,\dots ,\eta_{k+1})\in (S^2)^{k+1}$ and $r_1,\dots ,r_{k+1}>0$, let $x_i=r_i^k\eta_i$ and $X=(x_1,\dots ,x_{k+1})\in({\mathbb R}^3)^{k+1}$. Let  $g=g(X)\in M_{k+1}({\mathbb R}_3^+)$ have $(i,j)$-th entry
	$$g(X)_{ij}=|x_i|^k|x_j|^k G(\eta)_{ij}.$$ 
	Then 
	$$\det(\chi (g)(X))=|x_1|^{2k}\cdots |x_{k+1}|^{2k}\det(\chi (G)(\eta)).$$
	However, $\det (\chi (g)(X))$ is a polynomial on $({\mathbb R}^3)^{k+1}$ of degree $2k(k+1)$ and is homogeneous of degree $2k$ in every variable $x_i$. To see this observe that 
	\begin{align*}
		g(X)_{ij}&=|x_i|^k|x_j|^k G(\eta)_{ij}\\
		&=|x_i|^k|x_j|^k[(k+1)C_k^{\frac{1}{2}}(\langle\eta_i,\eta_j\rangle )+(\eta_i\wedge\eta_j)C_{k-1}^{\frac{3}{2}}(\langle\eta_i,\eta_j\rangle )]\\
		&=|x_i|^k|x_j|^k[\sum_{\ell =0}^{\lfloor\frac{k}{2}\rfloor}a_{k-2\ell}\langle\eta_i,\eta_j\rangle^{k-2\ell}+(\eta_i\wedge\eta_j)\sum_{\ell =0}^{\lfloor\frac{k-1}{2}\rfloor}b_\ell\langle\eta_i,\eta_j\rangle^{k-1-2\ell}]\\
		&=\sum_{\ell =0}^{\lfloor\frac{k}{2}\rfloor}a_{k-2\ell}\langle x_i,x_j\rangle^{k-2\ell}|x_i|^{2\ell}|x_j|^{2\ell}+(x_i\wedge x_j)\sum_{\ell =0}^{\lfloor\frac{k-1}{2}\rfloor}b_\ell\langle x_i,x_j\rangle^{k-1-2\ell}|x_i|^{2\ell} |x_j|^{2\ell}.
	\end{align*}
	So the entries of $\chi(g)(X)$ are homogeneous polynomials of degree $2k$, and $\det(\chi(g)(X))$ is a homogeneous polynomial of degree $2k(k+1)$ with complex coefficients.
	Furthermore, $\det (\chi (g)(X))=0$ on a set of infinite measure in $({\mathbb R}^3)^{k+1}$. To see this, let 
	$$E=\{\eta=(\eta_1,\eta_2,\dots ,\eta_{k+1})\in (S^2)^{k+1}:\, \det(\chi (G)(\eta))=0\}$$
	and
	\begin{align*}
	\tilde E&=\{(x_1,x_2,\dots ,x_{k+1})\in ({\mathbb R}^3)^{k+1}:\, x_i=r_i\eta_i \\
	& \qquad\qquad\qquad\qquad\qquad
	\text{ for some }r_1,\dots ,r_{k+1}>0\text{ and }(\eta_1,\dots ,\eta_{k+1})\in E\}.		
	\end{align*}
	Then $\det(\chi (g))(X)=0$ for all $X\in \tilde E$ and with $m_{k+1}$ denoting Lebesgue measure on $({\mathbb R}^3)^{k+1}$ and $\mu_{k+1}$ denoting the product surface measure on $(S^2)^{k+1}$, we have
	\begin{align*}
		m_{k+1}(\tilde E)&=\int_{({\mathbb R}^3)^{k+1}}\chi_{\tilde E}(x_1,\dots ,x_{k+1})\, dx_1\dots dx_{k+1}\\
		&=\int_0^\infty\cdots\int_0^\infty\int_{S^2}\cdots\int_{S^2}\chi_{\tilde E}(r_1\omega_1,\dots r_{k+1}\omega_{k+1})\, d\omega_1\cdots d\omega_{k+1}\\
		&\qquad\qquad\qquad r_1^2\cdots r_{k+1}^2\, dr_1\cdots\, dr_{k+1}\\
		&=\int_0^\infty\cdots\int_0^\infty\int_{S^2}\cdots\int_{S^2}\chi_{E}(\omega_1,\dots\omega_{k+1})\, d\omega_1\cdots d\omega_{k+1}\\
		&\qquad\qquad\qquad r_1^2\cdots r_{k+1}^2\, dr_1\cdots\, dr_{k+1}\\
		&=\mu_{k+1}(E)\int_0^\infty\cdots\int_0^\infty r_1^2\cdots r_{k+1}^2\, dr_1\cdots dr_{k+1}
	\end{align*}
	which is infinite. Since $\det((\chi (g))(x_1,\dots x_{k+1})$ is a polynomial in $x_1,\dots ,x_{k+1}$, this is only possible if $\det (\chi (g))(x_1,\dots ,x_{k+1})=0$ for all $(x_1,\dots ,x_{k+1})\in ({\mathbb R}^3)^{k+1}$, or equivalently, $\det (\chi (G)(\eta))=0$ for every choice of $\eta\subset (S^2)^{k+1}$ (for the proof see \cite{caron2005zero}). This implies that $G(\eta)$ is singular for each ensemble $\eta\in (S^2)^{k+1}$. Note that $G(\eta)=H(\eta)^*H(\eta)$ where $H(\eta)\in M_{k+1}(({\mathbb R}_3)^{k+1})$ has $(\ell, j)$-th entry $H(\eta)_{\ell j}=Y_k^\ell (\eta_j)$ for any choice $\{Y_k^\ell\}_{\ell =1}^{k+1}$ of orthonormal basis for ${\mathcal M}_k^+$, since
	\begin{align*}
	(H(\eta)^*H(\eta))_{ij}&=\sum_{\ell =1}^{k+1}H(\eta)_{i\ell}^*H(\eta)_{\ell j}=\sum_{\ell =1}^{k+1}\overline{H(\eta)_{\ell i}}H(\eta)_{\ell j}\\
	&=\sum_{\ell =1}^{k+1}\overline{Y_k^\ell (\eta_i)}Y_k^\ell (\eta _j)=K_k(\eta_i,\eta_j).
	\end{align*}
	Consequently, for each ensemble $\eta$ there exists ${\mathbf 0}\neq{\mathbf a}\in({\mathbb R}_3)^{k+1}$ such that $H(\eta)^*H(\eta){\mathbf a}=0$ and therefore
	$$0=[\langle H(\eta)^*H(\eta){\mathbf a},{\mathbf a}\rangle ]_0=\|H(\eta){\mathbf a}\|^2\Rightarrow H(\eta){\mathbf a}={\mathbf 0}$$
	so that $H(\eta)$ itself is singular. 
	Since we are at liberty to choose the orthonormal basis $\{Y_k^\ell\}_{\ell =0}^{k}$ for ${\mathcal M}_k^+({\mathbb R}^3)$, we let $Y_k^\ell =F_k^\ell$ where $\{F_k^\ell\}_{\ell =0}^{k}$ is the orthonormal basis constructed in section 3. 
	We want to show that $H(\eta)^*$ is also singular. The key observation is that the functions $F_k^\ell$ take values in ${\mathbb R}_3^+$, the even sub-algebra of ${\mathbb R}_3$ which is spanned by $\{1,e_1e_2,e_2e_3,e_1e_3\}$ and which is isomorphic to the quaternions. The mapping $H(\eta)^*:({\mathbb R}_3)^{k+1}\to ({\mathbb R}_3)^{k+1}$ has invariant subspaces $({\mathbb R}_3^+)^{k+1}$ and $({\mathbb R}_3^-)^{k+1}$. Let ${\mathbf 0}\neq {\mathbf a}\in ({\mathbb R}_3)^{k+1}$ be in the nullspace of $H(\eta)$, i.e., $H{\mathbf a}={\mathbf 0}$. We decompose ${\mathbf a}$ as ${\mathbf a}={\mathbf a}^++{\mathbf a}^-$ with ${\mathbf a}^+\in ({\mathbb R}_3^+)^{k+1}$ and ${\mathbf a}^-\in({\mathbb R}_3^+)^{k+1}$ (a unique decomposition) and observe that 
	$${\mathbf 0}=H(\eta){\mathbf a}=H(\eta){\mathbf a}^++H(\eta){\mathbf a}^-$$
	from which we conclude that $H(\eta){\mathbf a}^+={\mathbf 0}=H(\eta){\mathbf a^-}$. Since $({\mathbb R}_3^+)^{k+1}$ has no zero divisors, standard results of linear algebra apply. For example, with $u,v\in({\mathbb R}_3^+)^{k+1}$ we define the quaternion-valued inner product of $u$ and $v$ by
	$\langle u,v\rangle =\sum_{j=1}^{k+1}\overline{u_j}v_j$. 
	We can perform Gram-Schmidt orthogonalisation and define the orthogonal complement $X^\perp$ of submodules $X$ of $({\mathbb R}_3^+)^{k+1}$. Following the standard proofs, we have
	$$(\text{null\,} H(\eta)^*)=(\text{ran\,}(H(\eta)))^\perp\neq\{{\mathbf 0}\},$$
	i.e., for any $\eta\in (S^2)^{k+1}$ there exists ${\mathbf 0}\neq{\mathbf b}^+\in ({\mathbb R}_3^+)^{k+1}$  such that 
	$H(\eta)^*{\mathbf b}^+={\mathbf 0}$. We'll demonstrate that there is a an ensemble $\eta\in (S^2)^{k+1}$ for which $H(\eta)$ is invertible. In fact, we'll choose $\eta$ to be $k+1$ uniformly distributed points on the equator of $S^2$, i.e., let $\eta$ be the ensemble described in spherical coordinates by
	$$\eta=\{\eta_j=(\frac{2\pi j}{k+1},\frac{\pi}{2}):\, 0\leq j\leq k\}.$$
	Let $Y_k^n=\tilde F_k^n$ and with $H(\eta)_{ij}=\tilde F_k^i(\eta_j)$ we assume $H(\eta)^*{\mathbf b}={\mathbf 0}$ for some ${\mathbf 0}\neq{\mathbf b}\in ({\mathbb R}_3)^{k+1}$. Then
	$$0=(H(\eta)^*{\mathbf b})_i=\sum_{j=0}^k\overline{h_{ji}}b_j=\sum_{j=0}^k\tilde F_k^j(\eta_i)b_j.$$
	Hence, for $0\leq\ell\leq k$,
	\begin{align*}
		0&=\sum_{i=0}^k\overline{\tilde F_k^\ell (\eta_i)}\sum_{j=0}^k\tilde F_k^j(\eta_i)b_j\\
		&=\sum_{j=0}^k\left(\sum_{i=0}^k\overline{\tilde F_k^\ell (\eta_i)}\tilde F_k^j(\eta_i)\right)b_j\\
		&=\sum_{j=0}^k\frac{1}{2^{2\ell+2}(2k+3)\pi}[(k+1-\ell )^2P_{k+1-n}^{(\ell ,\ell )}(0)^2+(k+1)^2P_{k-\ell}^{(\ell ,\ell )}(0 )^2]\delta_{\ell j}b_j\\
		&=\frac{1}{2^{2n+2}(2k+3)\pi}[(k+1-\ell )^2P_{k+1-\ell }^{(\ell ,\ell )}(0)^2+(k+1)^2P_{k-\ell }^{(\ell ,\ell )}(0 )^2]b_\ell .
	\end{align*}
	However, $(k+1-\ell )^2P_{k+1-\ell}^{(\ell ,\ell)}(0)^2+(k+1)^2P_{k-\ell}^{(\ell ,\ell)}(0 )^2\neq 0$, for all $0\leq \ell \leq k$ and we conclude that $b_\ell =0$ for all $\ell$, i.e., ${\mathbf b}=0$, a contradiction. Hence $G(\eta)$ is invertible for a.e. choice of $\eta\in (S^2)^{k+1}$.
\end{proof}

Now, the rest of the process is similar to the harmonic version. We aim to find $k+1$ random points on the sphere such that the matrix $G$ is as close as possible to a diagonal matrix. From \eqref{matrix_G_for_finding_coefficients}, we see that the objective function for the monogenic case is
\begin{align}
	F(\eta_{1},\dots, \eta_{k+1})&=\sum_{i<j}\big\vert (k+1) (C_{k}^{\frac{1}{2}}(\langle\eta_{i},\eta_{j}\rangle))+ (\eta_{i}\wedge \eta_{j}) (C_{k-1}^{\frac{3}{2}}(\langle\eta_{i},\eta_{j}\rangle))\big\vert^{2}\nonumber\\
	&=\sum_{i<j}[(k+1)^{2} C_{k}^{\frac{1}{2}}(\langle\eta_{i},\eta_{j}\rangle)^2+(1-(\langle\eta_{i},\eta_{j}\rangle)^{2}) C_{k-1}^{\frac{3}{2}}(\langle\eta_{i},\eta_{j}\rangle)^{2}],\label{Objective_Function}
\end{align}
with $\{ \eta_{i} \}_{i=1}^{k+1}$ in $S^{2}$. Now we let 
\begin{align}
	G_{l}(t)&=F_{l}(\eta_{1},\dots,(\cos t) \,\eta_{l} + (\sin t) \,\vec{w},\dots,\eta_{k+1})\nonumber\\
	&=\sum_{j\neq l}\bigg[(k+1)^{2}\bigg(  C_{k}^{\frac{1}{2}}((\cos t) \langle \eta_{l}, \eta_{j}\rangle + (\sin t) \langle   \vec{w} ,\eta_{j}\rangle)\bigg)^2\nonumber\\
	&+\bigg( 1- ((\cos t)\langle {\eta}_{l} ,\eta_{j}\rangle + (\sin t)\langle \vec{w} , \eta_{j} \rangle)^{2}\bigg)\bigg( C_{k-1}^{\frac{3}{2}}((\cos t) \langle \eta_{l}, \eta_{j}\rangle + (\sin t) \langle   \vec{w} ,\eta_{j}\rangle)\bigg)^2\bigg]. \label{Objective_Function_2nd}
\end{align}
Differentiating both sides of \eqref{Objective_Function_2nd} with respect to $t$ gives
\begin{align}
	G_{l}'(t)&=\sum_{j\neq l}\bigg[(k+1)^{2} 2 C_{k}^{\frac{1}{2}}((\cos t) \langle \eta_{l}, \eta_{j}\rangle + (\sin t) \langle   \vec{w} ,\eta_{j}\rangle) \hspace*{3cm}\nonumber\\
	&\times (C_{k}^{\frac{1}{2}})'((\cos t) \langle \eta_{l}, \eta_{j}\rangle +(\sin t) \langle   \vec{w} ,\eta_{j}\rangle)(-(\sin t) \langle \eta_{l}, \eta_{j}\rangle + (\cos t) \langle   \vec{w} ,\eta_{j}\rangle)\nonumber\\
	&-2\, (\cos t\langle {\eta}_{l} ,\eta_{j}\rangle + (\sin t)\langle \vec{w} , \eta_{j} \rangle)(-(\sin t)\langle {\eta}_{l} ,\eta_{j}\rangle + (\cos t)\langle \vec{w} , \eta_{j} \rangle)\nonumber\\
	&\times \bigg(  C_{k-1}^{\frac{3}{2}}( (\cos t) \langle \eta_{l}, \eta_{j}\rangle + (\sin t) \langle   \vec{w} ,\eta_{j}\rangle)\bigg)^2\nonumber\\
	&+(1- ( (\cos t)\langle {\eta}_{l} ,\eta_{j}\rangle + (\sin t)\langle \vec{w} , \eta_{j} \rangle)^{2}) \, 2\, C_{k-1}^{\frac{3}{2}}( (\cos t) \langle \eta_{l}, \eta_{j}\rangle + (\sin t) \langle   \vec{w} ,\eta_{j}\rangle) \hspace*{3cm}\nonumber\\
	&\times (C_{k-1}^{\frac{3}{2}})'( (\cos t) \langle \eta_{l}, \eta_{j}\rangle + (\sin t) \langle   \vec{w} ,\eta_{j}\rangle)(-(\sin t) \langle \eta_{l}, \eta_{j}\rangle +  (\cos t) \langle   \vec{w} ,\eta_{j}\rangle)\bigg].\label{derivative_of_objective_function}
\end{align}
Putting $t=0$ in \eqref{derivative_of_objective_function} gives
\begin{align*}
	G'(t)\vert_{t=0}&= 2\bigg\langle \sum_{j\neq l}\bigg[(k+1)^{2} C_{k}^{\frac{1}{2}} \langle \eta_{l} , \eta_{j} \rangle (C_{k}^{\frac{1}{2}})' \langle \eta_{l} , \eta_{j} \rangle \eta_{j}\\ 
	&-\langle \eta_{l} , \eta_{j} \rangle \eta_{j} (C_{k-1}^{\frac{3}{2}} \langle \eta_{l} , \eta_{j} \rangle)^{2}\\ 
	&+ (1-\langle \eta_{l} , \eta_{j} \rangle^{2})  C_{k-1}^{\frac{3}{2}} \langle \eta_{l} , \eta_{j}\rangle (C_{k}^{\frac{3}{2}})' (\langle \eta_{l} , \eta_{j}\rangle) \eta_{j}
	\bigg], \vec{w}  \bigg \rangle. 
\end{align*}
Let
\begin{align*}
	w&=-\sum\limits_{j\neq l}\bigg[(k+1)^{2} C_{k}^{\frac{1}{2}} \langle \eta_{l} , \eta_{j} \rangle (C_{k}^{\frac{1}{2}})' \langle \eta_{l} , \eta_{j} \rangle 
	-\langle \eta_{l} , \eta_{j} \rangle  (C_{k-1}^{\frac{3}{2}} \langle \eta_{l} , \eta_{j} \rangle)^{2}\\ 
	&\qquad + (1-\langle \eta_{l} , \eta_{j} \rangle^{2})  C_{k-1}^{\frac{3}{2}} \langle \eta_{l} , \eta_{j}\rangle (C_{k-1}^{\frac{3}{2}})' (\langle \eta_{l} , \eta_{j}\rangle) 
	\bigg]\, (-\eta_{j}+\langle \eta_{j},\eta_{l} \rangle\, \eta_{l}).
\end{align*}
Then $w$ is the direction of the steepest descent of the objective function $F$ defined in \eqref{Objective_Function} when we vary $\eta_{l}$ only.
\begin{ex}
	For $k=2$ we will obtain the $3$ monogenic homogeneous functions $Z_{1},$ $Z_2$ and, $Z_{3}$. By using gradient descent on the sphere with random starting ensemble $\eta=(\eta_1,\eta_2,\eta_3)$ we find that a collection of points at which a local minimum of the objective function occurs is 
	\begin{align}
		\eta_{1}&=(0.4407, -0.1155, 0.8902),\nonumber\\
		\eta_{2}&=(-0.3322 ,  -0.7521  ,  0.5692),\label{Figure_three_optimized_points}\\
		\eta_{3}&=(0.5407 ,  -0.2516 ,  -0.8027).\nonumber
	\end{align}
	The objective function of those points takes the value $F(\eta_{1},\eta_{2},\eta_{3})=5.3999.$ Therefore, we can calculate 
	$A=\tau^{-1}\bigg(\chi^{-1}\bigg(\big(\chi(\tau G)\big)^{-\frac{1}{2}}\bigg)\bigg)$. The entries of $A$ are:
	\begin{align*}
		a_{11}&=(0.9123),\\
		a_{21}&=(0.1933)+(-0.1599)e_{12}+(0.2363)e_{13}+(0.261)e_{23},\\
		a_{31}&=(0.1933)+(0.02093)e_{12}+(0.361)e_{13}+(-0.1369)e_{23},
	\end{align*}
	
	\begin{align*}
		a_{12}&=(0.1933)+(0.1599)e_{12}+(-0.2363)e_{13}+(-0.261)e_{23},\\
		a_{22}&=(0.9123),\\
		a_{32}&=(0.0355)+(-0.212)e_{12}+(0.01778)e_{13}+(-0.3229)e_{23},
	\end{align*}
	
	\begin{align*}
		a_{13}&=(0.1933)+(-0.0795)e_{12}+(-0.361)e_{13}+(0.1369)e_{23},\\
		a_{23}&=(0.1933)+(0.212)e_{12}+(-0.01778)e_{13}+(0.3229)e_{23},\\
		a_{33}&=(0.9123).
	\end{align*}
	Now we put $Z_{i}(x)=\sum\limits_{j=1}^{3}K_{2}(x,\eta_{j})a_{ji}$. Note that the coefficient $a_{11}$ in the expansion of $Z_{1}$ is significantly larger than $a_{21},$ and $a_{31}$. In fact with $b=[0, a_{12},a_{13}]^{T}$, we have 
	\begin{align*}
		\Vert Z_{1}- K_{2}(\cdot,\eta_{1})a_{11}\Vert_{L^{2}(S^{m-1})}^{2}&=\Vert \sum_{j=1}^{3}K_{2}(\cdot,\eta_{j})a_{j1}-K_{2}(\cdot,\eta_{1})a_{11}\Vert_{L^{2}(S^{m-1})}^{2}\\
		&=\Vert K_{2}(\cdot,\eta_{2})a_{22}+K_{2}(\cdot,\eta_{3})a_{31}\Vert_{L^{2}(S^{m-1})}^{2}\\
		&=\Vert \sum_{j=1}^{3}K_{2}(\cdot,\eta_{j})b_{j}\Vert_{L^{2}(S^{m-1})}^{2}\\
		&= \bigg(\int_{S^{m-1}}\sum_{j=1}^{3} \overline{b_{j}} K_{2}(\cdot,\eta_{j}) \sum_{l=1}^{3} K_{2}(\cdot,\eta_{l}) b_{l}\; dx  \bigg)_{0}\\
		&= \bigg(\sum_{j,l=1}^{3} \overline{b_{j}} K_{2}(\eta_{j},\eta_{l}) b_{l}  \bigg)_{0}\\
		&= \big(\langle b,Gb \rangle \big)_{0}=0.8409 ,
	\end{align*}
	we conclude that that $Z_1$ is nearly zonal. Similarly, $a_{22}$ and $a_{33}$ are also the biggest coefficients in the expansions $Z_{2}(x),$ and $Z_{3}(x)$, respectively, and we conclude that $Z_2$, $Z_3$ are also nearly zonal. Codes for the computations of this example can be found at \href{https://github.com/HamedBaghal/Spherical\_Monogenics.git}{github.com/HamedBaghal/Spherical\_Monogenics.git}.
\end{ex}

\section*{Acknowledgment}
\noindent JAH is supported by the Australian Research Council through Discovery Grant DP160101537. Hamed has also the Lift-off fellowship from AUSTMS. Thanks, Roy. Thanks, HG. Thanks, AmirHosein.

\bibliographystyle{spmpsci}
\bibliography{Constructing_3_Dimensional_Monogenic_Homogeneous_Functions} 

\end{document}